\documentclass[a4paper,12pt,leqno]{amsart}
\usepackage{amsmath,amssymb,amsxtra,comment,graphicx,psfrag}
\usepackage{bm,mathrsfs}
\usepackage{mathtools}
\usepackage{xspace}
\usepackage{color}

\usepackage{enumerate}

\usepackage{hyperref}

\usepackage{pdfsync}

\usepackage{hhline}

\allowdisplaybreaks

\definecolor{blue}{rgb}{0,0.0,0.9}

 at 8. truept

\def\eps{\varepsilon}

\def\C{{\mathbb C}}
\def\R{{\mathbb R}}

\def\d{{\mathrm d}}
\def\e{{\rm e}}

\def\tr|{|\!|\!|}

\DeclareMathOperator\Real {Re}

\theoremstyle{plain}
\newtheorem{theorem}{Theorem}[section]
\newtheorem{lemma}[theorem]{Lemma}
\newtheorem{proposition}[theorem]{Proposition}

\theoremstyle{definition}

\newtheorem{remark}{Remark}[section]

\numberwithin{equation}{section}
\numberwithin{table}{section}
\numberwithin{figure}{section}

\begin{document}

\title[Maximal parabolic regularity and energy estimates]
{Combining maximal regularity and energy estimates for
 time discretizations\\
  of quasilinear parabolic equations} 

\author[Georgios Akrivis]{Georgios Akrivis}
\address{Department of Computer Science \& Engineering, University of Ioannina, 451$\,$10
Ioannina, Greece} 
\email {\sf{akrivis{\it @}cs.uoi.gr} }

\author[Buyang Li]{Buyang Li}
\address{Department of Applied Mathematics, The Hong Kong Polytechnic University, Kowloon, Hong Kong.} 
\email {\sf{buyang.li{\it @}polyu.edu.hk} }

\author[Christian Lubich]{Christian Lubich}
\address{Mathematisches Institut, Universit\"at T\"ubingen, Auf der Morgenstelle, 
D-72076 T\"ubingen, Germany} 
\email {\sf{lubich{\it @}na.uni-tuebingen.de} }



\date{\today}

\begin{abstract}
We analyze fully implicit and linearly implicit backward difference formula (BDF) methods
for quasilinear parabolic equations, without making any assumptions on the growth  
or decay of the coefficient functions. We combine maximal parabolic regularity and 
energy estimates to derive optimal-order error bounds for the time-discrete approximation 
to the solution and its gradient in the maximum norm and energy norm.
\end{abstract}

\maketitle

\section{Introduction}\label{Sec:intr}
In this paper we study the time discretization of quasilinear parabolic differential equations 
by backward difference formulas (BDF). In contrast to the existing literature, we allow for 
solution-dependent coefficients in the equation that degenerate as the argument grows to 
infinity or approaches a singular set. To deal with such problems, we need to control the 
maximum norm of the error and possibly also of its gradient. As we show in this paper, 
such maximum norm estimates for BDF time discretizations become available by combining 
two techniques:
\begin{enumerate}[\small $\bullet$]\itemsep=0pt
\item \emph{discrete maximal parabolic regularity}, as studied in 
Kov\'acs, Li \& Lubich \cite{KLL} based on the characterization of maximal $L^p$-regularity 
by Weis \cite{Weis2} and a discrete operator-valued Fourier multiplier theorem by Blunck
 \cite{Blu1}; and 
\item \emph{energy estimates}, which are familiar for implicit Euler time discretizations 
and have  become feasible for higher-order BDF methods (up to order 5) by the 
Nevan\-linna-Odeh multiplier technique \cite{NO} as used in Akrivis \& Lubich \cite{AL}.
\end{enumerate}
In Section~\ref{Sec:statement} we formulate the parabolic initial and boundary value problem 
and its time discretization by fully implicit and linearly implicit BDF methods, and we state 
our main results. For problems on a bounded Lipschitz domain $\varOmega$ in $\R^2$ 
or $\R^3$ with sufficiently regular solutions, we obtain optimal-order error bounds in the 
maximum and energy norms. For problems on bounded smooth domains in arbitrary space
dimension we obtain optimal-order error bounds even in the 
$L^\infty(0,T;W^{1,\infty}(\varOmega))$ and $L^2(0,T;H^2(\varOmega))$ norms.

The proof of these results makes up the remainder of the paper. In Section~\ref{Sec:abstract} 
we formulate a common abstract framework for both results and give a continuous-time
perturbation result after which we model the stability proof of BDF methods in 
Section~\ref{Sec:stability}. This proof relies on time-discrete maximal regularity and 
energy estimates. In Section~\ref{Sec:existence} we study existence and uniqueness 
of the numerical solutions, and in Section~\ref{Sec:consistency} we discuss the consistency 
error of the fully and linearly implicit BDF methods. In the short Section~\ref{Sec:proof} we 
combine the obtained estimates to prove the error bounds of our main results. In the 
remaining Sections~\ref{Sec:unif-maxreg} and \ref{Sec:Sobolev}, which use different 
techniques of analysis, it is verified that the concrete parabolic problem and its time 
discretization fit into our abstract framework. In particular, the uniform discrete maximal 
parabolic regularity of the BDF methods is shown in the required $L^q$ and $W^{-1,q}$ 
settings.

While we restrict our attention in this paper to semidiscretization in time of standard 
quasilinear parabolic equations by BDF methods, the combination of discrete maximal 
regularity and energy estimates to obtain stability and error bounds in the maximum norm 
is useful for a much wider range of problems: for other time discretizations,  for full 
discretizations  (in combination with the discrete maximal regularity of semi-discrete 
finite element methods; see Li \cite{Li15} and Li \& Sun \cite{LS15}), and also for other 
classes of nonlinear parabolic equations. 
Such extensions are left to future work. The present paper can thus be viewed as a 
proof of concept for this powerful approach.

\section{Problem formulation and statement of the main results}\label{Sec:statement}
\subsection{Initial and boundary value problem}\label{SSec:ivp}
For a bounded domain $\varOmega\subset \R^d$, a positive $T,$
and a given initial value $u_0,$ 
we consider the following initial and boundary value problem 
for a quasilinear parabolic equation, with homogeneous Dirichlet boundary conditions, 
\begin{equation}
\label{ivp}
\left \{
\begin{alignedat}{3} 
&\partial_t u =\nabla\cdot \big (a(u)\nabla u) \quad && \text{in }\ &&\varOmega\times [0,T],\\
&u =0 \quad && \text{on }\ &&\partial \varOmega\times [0,T],\\
&u (\cdot,0)=u_0 && \text{in }\  &&\varOmega.\\
\end{alignedat}
\right .
\end{equation}
We assume that $a$ is a positive smooth function on the real line,
but impose otherwise no growth or decay 
conditions on $a$ (for example, we may have $a(u)=e^u$).
By the maximum principle, 
the solution of the above problem is bounded, 
provided the initial data $u_0$ is bounded; 
note that then there exists a positive number $K$ (depending on 
$\|u\|_{L^\infty(0,T;L^\infty( \varOmega))}$) such that $K^{-1}\le a(u(x,t))\le K$, 
for all $x\in \overline\varOmega$ and $0\le t\le T$. However, boundedness of the 
numerical approximations is not obvious.

\begin{remark}[direct extensions]\label{Re:ext}
Our techniques and results can be directly extended to the following cases:
\begin{enumerate}[\small $\bullet$]\itemsep=0pt
\item The function $a$ is continuously differentiable and positive only in an interval $I\subset \R$
that contains all exact solution values, $u(x,t)\in I;$ in particular, singularities
in $a$ are allowed.
\item $a$ is a function of $x,t$ and $u$: $a=a(x,t,u).$
\item $a(u)$ is a positive definite symmetric $d\times d$ matrix.
\item A semilinear term $f(x,t,u,\nabla u)$ with a smooth function $f$ is added to the right-hand 
side of the differential equation. No growth conditions on $f$ need to be imposed, but we 
assume smoothness of the exact solution.
\end{enumerate}
\end{remark}

\emph{Operator notation}: We consider $A(w)$ defined by 
$A(w)u:=-\nabla\cdot\big (a(w)\nabla u)$ as a linear
operator on $L^q(\varOmega),$ self-adjoint on $L^2(\varOmega).$

\subsection{Fully and linearly implicit BDF methods}\label{SSec:bdf}

\subsubsection{Fully implicit methods}\label{SSSec:fully-im}
We let $t_n=n\tau, n=0,\dotsc,N,$ be a uniform partition of the interval
$[0,T]$ with time step $\tau=T/N,$ and consider general 
$k$-step backward difference formulas (BDF) for the discretization of \eqref{ivp}:
\begin{equation}
\label{bdf:fully-im1}
\frac 1\tau \sum_{j=0}^k\delta_ju_{n-j}=-A(u_n)u_n,
\end{equation}
for $n=k,\dotsc,N,$ where $u_1,\dotsc,u_{k-1}$ are sufficiently accurate given starting 
approximations and the coefficients of the method are given by
\[\delta(\zeta)=\sum_{j=0}^k\delta_j\zeta^j=\sum_{ \ell=1}^k\frac 1 \ell (1-\zeta)^ \ell.\]
The method is known to have order $k$ and to be A$(\alpha)$-stable with angle
$\alpha=90^\circ,90^\circ,86.03^\circ,73.35^\circ, 51.84^\circ,17.84^\circ$
for $k=1,\dotsc,6,$ respectively; see \cite[Section V.2]{HW}.   
A$(\alpha)$-stability is equivalent to $|\arg\delta(\zeta)|\le \pi-\alpha$ for $|\zeta|\le 1.$
Note that the first- and second-order BDF methods are A-stable, that is, 
$\Real \delta(\zeta)\ge 0$ for $|\zeta|\le 1.$ 

\subsubsection{Linearly implicit methods}\label{SSSec:lin-im}
Since equation \eqref{bdf:fully-im1} is in general nonlinear in the unknown $u_n,$ 
we will also consider the following linearly implicit modification:
\begin{equation}
\label{bdf:lin-im1}
\frac 1\tau \sum_{j=0}^k\delta_ju_{n-j}=
-A\Big (\sum\limits^{k-1}_{j=0}\gamma_ju_{n-j-1}\Big )u_n,
\end{equation}
for $n=k,\dotsc,N,$ with the coefficients $\gamma_0,\dotsc,\gamma_{k-1}$ given by
\[\gamma (\zeta)=\frac 1  \zeta\big [1-(1-\zeta)^k\big ]=\sum_{i=0}^{k-1} \gamma_i\zeta^i.\]
Notice that now the unknown $u_n$ appears in \eqref{bdf:lin-im1} only linearly;
therefore, to advance with \eqref{bdf:lin-im1} in time, we only need to solve, at each time level, 
just one linear equation, which reduces to a linear system if we discretize also in space.

\subsection{Main results}\label{Sec:result}

In this paper we prove the following two results. 

\begin{theorem}\label{Th:main-1}
Let  $\varOmega\subset\R^d$, $d=2,3$, be a bounded Lipschitz domain.  
If the solution $u$ of \eqref{ivp} is sufficiently regular and the initial approximations 
are sufficiently accurate, then there exist ${\tau_0>0}$ and $C<\infty$ such that for 
stepsizes $\tau\le\tau_0$ and $N\tau\le T$, the  fully and linearly implicit BDF methods
 \eqref{bdf:fully-im1} and \eqref{bdf:lin-im1}, respectively, of order $k\le 5$, have unique 
 numerical solutions $u_n\in C(\overline \varOmega)\cap H^1_0(\varOmega)$ with 
 errors bounded by
\begin{align}
\label{est-C1}
\max_{k\leq n\leq N}
\|u_n-u(t_n)\|_{L^\infty(\varOmega)}
\le C\tau^k,
\\[1mm]
\label{est-H2}
\Bigl( \tau \sum_{n=k}^N \|u_n-u(t_n)\|_{H^1(\varOmega)}^2 \Bigr)^{1/2}\le C \tau^k.
\end{align}
\end{theorem}

\begin{theorem}\label{Th:main-2}
Let the bounded domain $\varOmega\subset\R^d$ be smooth, 
where $d\geq 1$.  
If the solution $u$ of \eqref{ivp} is sufficiently regular and the initial approximations 
are sufficiently accurate, then there exist ${\tau_0>0}$ and $C<\infty$ such that for 
stepsizes $\tau\le\tau_0$ and $N\tau\le T$, the  fully and linearly implicit BDF methods
 \eqref{bdf:fully-im1} and \eqref{bdf:lin-im1}, respectively, of order $k\le 5$, have 
 unique numerical solutions 
 $u_n\in C^1(\overline \varOmega)\cap H^2(\varOmega)\cap H^1_0(\varOmega)$ 
 with errors bounded by
\begin{align}
\label{est-C1-smooth}
\max_{k\leq n\leq N}
\bigl(\|u_n-u(t_n)\|_{L^\infty(\varOmega)}
+\|\nabla u_n-\nabla u(t_n)\|_{L^\infty(\varOmega)} \bigr)
\le C\tau^k,
\\
\label{est-H2-smooth}
\Bigl( \tau \sum_{n=k}^N \|u_n-u(t_n)\|_{H^2(\varOmega)}^2 \Bigr)^{1/2}\le C \tau^k.
\end{align}
\end{theorem}

Let us first comment on the regularity requirements. 
With some $q>d$, we need to assume in Theorem~\ref{Th:main-1}
\begin{equation}\label{reg-1}
u\in C^{k+1}\big ([0,T];W^{-1,q}(\varOmega)\big )
\cap C^{k}\big ([0,T];L^q(\varOmega)\big )\cap C\big ([0,T];W^{1,q}(\varOmega)\big )  ,
\end{equation}
and in Theorem~\ref{Th:main-2}
\begin{equation}\label{reg-2}
u\in C^{k+1}\big ([0,T];L^q(\varOmega)\big )
\cap C^{k}\big ([0,T];W^{1,q}(\varOmega)\big )\cap C\big ([0,T];W^{2,q}(\varOmega)\big ) .
\end{equation}
The errors in the initial data $e_n=u_n-u(t_n)$, for $n=0,\dotsc,k-1$, need to satisfy 
the following bounds:  in Theorem~\ref{Th:main-1},
\begin{equation}
\label{InitialError}
\bigg (\tau \sum_{n=1}^{k-1} 
\Big\|\frac{e_n-e_{n-1}}{\tau}\Big\|^p_{W^{-1,q}(\varOmega)}\bigg )^{\frac{1}{p}}
+\bigg (\tau \sum_{n=1}^{k-1} \|e_n\|^p_{W^{1,q}(\varOmega)}\bigg )^{\frac{1}{p}}
\le C\tau^k,
\end{equation}
for some $p$ such that $2/p+d/q<1$, and similarly in Theorem~\ref{Th:main-2} with 
\begin{equation}
\label{InitialError2}
\bigg (\tau \sum_{n=1}^{k-1} 
\Big\|\frac{e_n-e_{n-1}}{\tau}\Big\|^p_{L^q(\varOmega)}\bigg )^{\frac{1}{p}}
+\bigg (\tau \sum_{n=1}^{k-1} \|e_n\|^p_{W^{2,q}(\varOmega)}\bigg )^{\frac{1}{p}}
\le C\tau^k.
\end{equation} 
It can be shown that these bounds are satisfied when the starting values are obtained 
with an algebraically stable implicit Runge--Kutta method of stage order $k$, such as 
the $(k-1)$-stage Radau collocation method.

Error bounds for BDF time discretizations of quasilinear parabolic differential equations 
have previously been obtained by Zl\'amal \cite{Zla}, for $k\le 2$, and by Akrivis \& Lubich 
\cite{AL} for $k\le 5$, using energy estimates. 
Implicit--explicit multistep methods for a class of such equations have been analyzed 
by Akrivis, Crouzeix \& Makridakis \cite{ACM2} by spectral and Fourier techniques. 
In those papers it is, however, assumed 
that the operators $A(u)$ are uniformly elliptic for $u\in H^1_0(\varOmega)$, which amounts 
to assuming that the coefficient function $a$ is bounded on all $\R$ and has a strictly 
positive lower bound on all $\R$. This is a restrictive assumption that is not satisfied 
in many applications.

This restriction can be overcome only by controlling the maximum norm of the numerical 
solution, which is a major contribution of the present paper. 
Since no maximum principle is available for the BDF methods of order higher than 1, 
the boundedness of the numerical solution  is not obvious. While there are some results 
on maximum norm error bounds for implicit  Euler and Crank--Nicolson time discretizations 
of \emph{linear} parabolic equations by  Schatz, Thom\'ee \& Wahlbin 
\cite{STW1}, 
we are not aware of any such results for quasilinear parabolic equations 
as studied here.

In our view, even more interesting than the above particular results is the novel technique 
by which they are proved: by combining \emph{discrete maximal regularity} and 
\emph{energy estimates}.
The combination of these techniques will actually yield $O(\tau^k)$ error bounds in 
somewhat stronger norms than stated in Theorems~\ref{Th:main-1} and \ref{Th:main-2}. 
Moreover, we provide a concise abstract framework in which the combination of maximal 
regularity and energy estimates can be done and which allows for a common proof for 
both Theorems~\ref{Th:main-1} and \ref{Th:main-2}, as well as for direct extensions to 
more general quasilinear parabolic problems than \eqref{ivp}.

\section{Abstract framework and basic approach in continuous time}
\label{Sec:abstract}
As a preparation for the proof of Theorems~\ref{Th:main-1} and \ref{Th:main-2}, it  is helpful 
to illustrate the approach taken in this paper first in a time-continuous and more abstract 
setting, which in particular applies to $A(w)u=-\nabla\cdot (a(w) \nabla u)$ as considered above.

\subsection{Abstract framework} \label{subsec:abstract-framework}
We formulate an abstract setting that works with Hilbert spaces $V\subset H$ and 
Banach spaces $D\subset W \subset X $ 
as follows: Let $H$ be the basic Hilbert space, and let $V$ be another Hilbert space that is 
densely and continuously imbedded in $H$. Together with the dual spaces we then have 
the familiar Gelfand triple of Hilbert spaces $V \subset H=H' \subset V'$ with dense and 
continuous imbeddings, and such that the restriction to $V\times H$ of the duality pairing 
$\langle\cdot,\cdot\rangle$ between $V$ and $V'$ and of the inner product $(\cdot,\cdot)$ 
on $H$ coincide. Let further $X \subset V'$ be a Banach space and let $D\subset W\subset X$ 
be further Banach spaces. We denote the corresponding norms by $\|\cdot \|_H$, 
$\|\cdot \|_V$, $\|\cdot \|_X$,  $\|\cdot \|_W$, and $\|\cdot\|_D$, respectively, and summarize the 
continuous imbeddings:
%
\[\begin{array}{ccccc}
V & \subset & H & \subset & V' \\
\cup &         &  \cup   &             & \cup \\
D &\subset & W  &\subset & X
\end{array}\]
For the existence of the numerical solution of the fully implicit BDF method we will 
further require that  $D$ is compactly imbedded in $W$.

We have primarily the following two situations in mind:
\begin{enumerate}[(P1)]\itemsep=0pt
\item For a bounded Lipschitz domain $\varOmega\subset\R^d$ (with $d\le 3$) 
we consider the usual 
Hilbert spaces $H=L^2(\varOmega)$ and $V=H^1_0(\varOmega)$, and in addition the 
Banach spaces $X=W^{-1,q}(\varOmega)$ for suitable $q>d$,
$W=C^\alpha(\overline \varOmega)$ with a small $\alpha>0$, and $D=W^{1,q}(\varOmega)\cap H^1_0(\varOmega)$.
\item For a smooth bounded domain $\varOmega\subset\R^d$ (with arbitrary dimension $d$) we consider again 
the  Hilbert spaces $H=L^2(\varOmega)$ and $V=H^1_0(\varOmega)$, and the Banach 
spaces $X=L^q(\varOmega)$ with $q>d$, $W=C^{1,\alpha}(\overline\varOmega)$ 
with a small $\alpha>0$, and $D=W^{2,q}(\varOmega)\cap H^1_0(\varOmega)$.
\end{enumerate}

We will work with the following five conditions:

(i) (\emph{$W$-locally uniform maximal regularity}) For $w\in W$,  the  linear operator 
$-A(w)$ is the generator of an analytic semigroup on $X$ with domain $D(A(w))=D$ 
independent of~$w$, and has maximal $L^p$-regularity: for $1<p<\infty$ there exists 
a real $C_{p}(w)$ such that the solution of the inhomogeneous initial value problem
\begin{equation}
\label{lin-eq-Av}
\dot u(t) + A(w)u(t) = f(t)\qquad (0<t\le T), \qquad u(0)=0 ,
\end{equation}
is bounded by
\[\| \dot u \|_{L^p(0,T;X)} + \| A(w) u \|_{L^p(0,T;X)} \le C_{p}(w) \| f \|_{L^p(0,T;X)}
\qquad \forall f \in {L^p(0,T;X)}.\]
Moreover, the bound is uniform in bounded sets of $W$: for every $R>0$,
\[C_{p}(w) \le C_{p,R} \quad\text{if }\ \|w\|_{W} \le R.\]
We further require that the graph norms $\| \cdot \|_X + \| A(w)\cdot \|_X$ are uniformly 
equivalent to the norm $\|\cdot\|_D$ for $\|w\|_{W} \le R$.

(ii) (\emph{Control of the $W$-norm by maximal regularity})
For some $1<p<\infty$, we have a continuous imbedding 
$W^{1,p}(0,T;X) \cap L^p(0,T;D)\subset L^\infty(0,T;W)$: there is $C_p<\infty$ 
such that  for all $u \in W^{1,p}(0,T;X) \cap L^p(0,T;D)$,
\[\| u \|_{L^\infty(0,T;W)} \le C_p \Bigl( \| \dot u \|_{L^p(0,T;X)} + \| u \|_{L^p(0,T;D)} \Bigr).\]

(iii) (\emph{$V$-ellipticity}) $A(w)$ extends by density to a bounded linear operator 
$A(w):V \to V'$, and for all $w\in W$ with $W$-norm bounded by $R$, the operator 
$A(w)$ is uniformly $V$-elliptic:
\[\alpha_R \| u \|_{V}^2 \le \langle u, A(w) u \rangle \le M_R \| u \|_{V}^2
\qquad \forall u \in V\]
with $\alpha_R>0$ and  $M_R<\infty$, where $\langle\cdot,\cdot\rangle$ denotes the duality 
pairing between $V$ and~$V'$. 

(iv) (\emph{Operators with different arguments: $X$-norm estimate}) 
For every $\eps>0$, there is $C_{\eps,R}<\infty$ such that for all $v,w\in W$ that are 
bounded by $R$ in the $W$-norm, and for all $u\in D$,
\[\| (A(v)-A(w))u \| _{X} \le \Bigl( \eps\, \| v-w \|_{W} + C_{\eps,R}\, \| v-w \|_{H} \Bigr) \|u\|_D.\]

(v) (\emph{Operators with different arguments: $V'$-norm estimate}) For every $\eps>0$, 
there is $C_{\eps,R}<\infty$ such that for all $v,w\in W$ that are bounded by $R$ in the 
$W$-norm, and for all $u\in D$,
\[ \| (A(v)-A(w))u\|_{V'} \le 
\Bigl( \eps \| v-w \|_{V} + C_{\eps,R} \,\| v-w \|_{H} \Bigr)\, \|u \|_{D} .\]
%

\begin{lemma} \label{lem:framework}
The operators given by $A(w)u=-\nabla\cdot (a(w) \nabla u)$ with homogeneous 
Dirichlet boundary conditions and a smooth positive function $a(\cdot)$ satisfy 
\emph{(i)--(v)} in the situations \emph{(P1)} and \emph{(P2)} above.
\end{lemma}
The proof of this lemma will be given in Section 
\ref{Sec:unif-maxreg} and Section \ref{Sec:Sobolev}.
In the cases (P1) and (P2) we actually have a stronger bound than (v):
\[\| (A(v)-A(w))u\|_{V'} \le C_{R}\,\| v-w \|_{H} \, \|u \|_{D} .\]

\subsection{A perturbation result}

Suppose now that $u\in W^{1,p}(0,T;X)\cap L^p(0,T;D)$ solves
\begin{equation}
\label{ivp-numer}
\left \{
\begin{aligned} 
&\dot u (t) + A(u (t))u (t)=f(t), \qquad 0<t\le T,\\
&u (0)=u_0,
\end{aligned}
\right .
\end{equation}
and $u^\star\in W^{1,\infty}(0,T;W)\cap L^p(0,T;D)$ solves the perturbed equation
%
\begin{equation}
\label{ivp2}
\left \{
\begin{aligned} 
&\dot u^\star (t) + A(u^\star (t))u^\star(t)=f(t)+d(t), \quad 0<t\le T,\\
&u^\star(0)=u_0,
\end{aligned}
\right .
\end{equation}
where the defect $d$ is  bounded by
\begin{equation}
\label{est-delta}
\|d\|_{L^p(0,T;X)}\le \delta. 
\end{equation} 

As an illustration of the combined use of maximal $L^p$-regularity and energy estimates 
we prove the following result. Theorems~\ref{Th:main-1} and \ref{Th:main-2} will be 
proved by transferring these arguments to the time-discrete setting.

\begin{proposition}
\label{prop:basic} In the above setting of {\rm(i)--(v)} and \eqref{ivp-numer}--\eqref{est-delta} 
with $\delta>0$ sufficiently small, 
the error $e=u-u^\star$ between the solutions of  \eqref{ivp-numer} and \eqref{ivp2}
is bounded by
\begin{align*}
\| \dot e \|_{L^p(0,T;X)} + \| e \|_{L^p(0,T;D)} &\le C\delta
\\
\| e \|_{L^\infty(0,T;W)}  &\le C\delta,
\end{align*}
where $C$ depends on 
$\| u^\star \|_{W^{1,\infty}(0,T;W)}$ 
and $\| u^\star \|_{L^p(0,T;D)}$, 
but is independent of $\delta$.
\end{proposition}

\begin{proof} (a) (\emph{Error equation})
We rewrite the equation in \eqref{ivp-numer} in the form
\[\dot u (t) + A(u^\star (t))u (t)=\big (A(u^\star (t))-A(u (t))\big )u (t)+f(t)\]
and see that the error $e=u-u^\star$ satisfies the {error equation}
\begin{equation}
\label{er-eq}
\dot e (t) + A(u^\star (t))e (t)=\big (A(u^\star (t))-A(u (t))\big )u (t)-d(t).
\end{equation}
Obviously, $e(0)=0.$
To simplify the notation, we denote  $\bar A(t):=A(u^\star (t)),$ and rewrite the error equation 
with some arbitrary $\bar t\ge t$ as
\begin{equation}
\label{er-eq2}
\dot e (t) + \bar A(\bar t)e(t)=\big ( \bar A(\bar t)-\bar A(t)\big )e(t)
+\big (A(u^\star (t))-A(u (t))\big )u (t)-d(t),
\end{equation}
i.e.,
\begin{equation}
\label{er-eq3}
\dot e (t) + \bar A(\bar t)e(t)=\widehat d(t),
\end{equation}
with
\begin{equation}
\label{er-eq4}
\widehat d(t):=\big ( \bar A(\bar t)-\bar A(t)\big )e(t)-\big (A(u^\star (t))-A(u (t))\big )u (t)-d(t).
\end{equation}
(b) (\emph{Maximal regularity})
We denote 
\[R=\| u^\star \|_{L^\infty(0,T;W)}+1\quad\text{and}\quad
B=  \| u^\star \|_{L^p(0,T;D)}+1 \]
and let $0<t^*\le T$ be maximal such that 
\begin{equation} \label{t-star}
\| u \|_{L^\infty(0,t^*;W)} \le R \quad\text{and}\quad
\| u \| _{L^p(0,t^*;D)} \le B.
\end{equation}
By the maximal $L^p$-regularity (i)
we immediately obtain from \eqref{er-eq3} for $\bar t\le t^*$
\begin{equation}
\label{max-reg}
\|\dot e\|_{L^p(0,\bar t;X)}+\| e\|_{L^p(0,\bar t;D)}\le C_{p,R}\, \|\widehat d\|_{L^p(0,\bar t;X)}.
\end{equation}
By the bound (iv) and the assumed Lipschitz continuity of $u^\star:[0,T]\to W$, we have for 
any $\eps>0$
\begin{align}
\label{d-hat0}
\| \widehat d(t) \|_{X} & \le C (\bar t  - t) \| e(t) \|_{D} 
\\
\nonumber
&+ \eps\, \| e(t) \|_{W} \,  \| u(t) \|_{D} + C_{\eps,R}\, \| e(t) \|_{H} \,  \| u(t) \|_{D}  + \|  d(t) \|_{X} .
\end{align}
We take the second term on the left-hand side of \eqref{max-reg} to power $p$ and denote it by
\[\eta( t) = \| e\|_{L^p(0, t;D)}^p\,.\]
For the first term on the right-hand side of \eqref{d-hat0} we note that, by partial integration,
\[\int_0^{\bar t} (\bar t  - t)^p \,\| e(t) \|_{D}^p\, d t = p \int_0^{\bar t}  (\bar t  - t)^{p-1} \eta(t)\, dt.\]
Hence we have from \eqref{max-reg}
\begin{align*}
\eta(\bar t) \le C_{p,R}^p \| \widehat d \|_{L^p(0,\bar t;X)} ^p &\le 
C \int_0^{\bar t}  (\bar t  - t)^{p-1} \eta(t)\, dt 
\\
&+ C\Bigl( \eps B\| e  \|_{L^\infty(0,\bar t;W)} + C_{\eps,R} B \| e\|_{L^\infty(0,\bar t;H)}  
+ C\|d\|_{L^p(0,\bar t;X)}\Bigr)^p.
\end{align*}
With a Gronwall inequality, we therefore obtain from \eqref{max-reg}
\[\|\dot e\|_{L^p(0,\bar t;X)}+\| e\|_{L^p(0,\bar t;D)}\le C \eps \| e  \|_{L^\infty(0,\bar t;W)} 
+ C_{\eps,R,B}  \| e\|_{L^\infty(0,\bar t;H)}  + C\|d\|_{L^p(0,\bar t;X)},\]
and we note that by property (ii) the left-hand term dominates $\| e  \|_{L^\infty(0,\bar t;W)}$.
For sufficiently small $\eps$ we can therefore absorb the first term of the right-hand side 
in the left-hand side to obtain
\begin{equation}
\label{max-reg-2}
\|\dot e\|_{L^p(0, t^*;X)}+\| e\|_{L^p(0,t^*;D)}\le C \| e \|_{L^\infty(0,t^*;H)}+
C \delta.
\end{equation}

(c) (\emph{Energy estimate}) To bound the first term on the right-hand side of \eqref{max-reg-2} 
we use the energy estimate obtained by testing \eqref{er-eq} with $e$:
\[\frac12 \frac d {dt} \| e(t) \|_H^2 + \langle e(t), A(u^\star(t))e(t) \rangle 
= \bigl\langle e(t), (A(u^\star(t) )-A(u(t) ))u(t)  \bigr\rangle - \langle e(t),d(t) \rangle.\]
The  bound of (v) yields 
\[\bigl\langle e, (A(u^\star )-A(u ))u  \bigr\rangle \le \| e \|_{V} \, \Bigl( \eps \| e \|_V + C_{\eps,R} 
\| e \|_{H} \Bigr) \|u\|_{D}\le
\Bigl( 2\eps \| e \|_{V}^2 + C_{\eps,R}  \| e \|_{H}^2 \Bigr)  \|u\|_{D}.\]
Integrating from $0$ to $t\le t^*$, using the $V$-ellipticity (iii)  and absorbing the term 
with $\| e \|_{V}^2$ we therefore obtain 
\[ \| e(t) \|_H^2 + \int_0^t  \| e(s) \|_V^2\, ds \le C \int_0^t \| e(s) \|_{H}^2\, ds
  + C \int_0^t \| d(s)\|_{V'}^2\, ds,\]
 and the Gronwall inequality then yields
\[\|e(t)\|_{H} \le C\delta, \qquad 0\le t \le t^*.\]

(d) (\emph{Complete time interval}) Inserting the previous bound in \eqref{max-reg-2}, we obtain
\[\|\dot e\|_{L^p(0, t^*;X)}+\| e\|_{L^p(0,t^*;D)}\le C \delta,\]
which by (ii)  further implies
\[\| e\|_{L^\infty(0,t^*;W)} \le C \delta.\]
Hence, for sufficiently small $\delta$ we have strict inequalities
\[\| u \|_{L^\infty(0,t^*;W)} < R \quad\text{and}\quad
\| u \| _{L^p(0,t^*;D)} < B.\]
In view of the maximality of $t^*$ with \eqref{t-star} this is possible only if $t^*=T$.
\end{proof}

\begin{remark}
 If condition (ii) is strengthened to
 \[\| u \|_{C^\alpha([0,T];W)} \le C_{\alpha,p} \Bigl( \| \dot u \|_{L^p(0,T;X)} 
 + \| u \|_{L^p(0,T;D)} \Bigr),\]
with some $\alpha>0$, then the statement of Proposition~\ref{prop:basic} remains valid 
under the weaker condition $u,u^\star\in W^{1,p}(0,T;X)\cap L^p(0,T;D)$, which is 
symmetric in $u$ and~$u^\star$. The proof remains essentially the same.
\end{remark}

\section{Stability estimate for BDF methods}
\label{Sec:stability}

\subsection{Abstract framework for the time discretization}
We work again with the abstract framework (i)--(v) of the previous section and consider 
in addition the following property of the BDF time discretization. Here we denote for a 
sequence $(v_n)_{n=1}^N$ and a given stepsize $\tau$
\[ \big\|(v_n )_{n=1}^N\big\|_{L^p(X)} = \Bigl( \tau \sum_{n=1}^N \| v_n \|_X^p \Bigr)^{1/p},\]
which is the $L^p(0,N\tau;X)$ norm of the piecewise constant function taking the 
values~$v_n$. We will work with the following discrete analog of condition (i).

(i') (\emph{$W$-locally uniform discrete maximal regularity}) For $w\in W$,  the  linear 
operator $-A(w)$  has discrete maximal $ L^p$-regularity for the BDF method: for 
$1<p<\infty$ there exists a real $C_{p}(w)$ such that for every stepsize $\tau>0$ 
the numerical solution determined by
\begin{equation}
\label{lin-eq-Av-bdf}
\dot  u_n + A(w)u_n = f_n\qquad (k\le n\le N) \qquad\text{with }\  
\dot u_n = \frac1\tau \sum_{j=0}^k \delta_j u_{n-j}
\end{equation}
for starting values $u_0=\dotsb=u_{k-1}=0$, is bounded by
    \begin{equation*}
        \big\|(\dot u_n )_{n=k}^N\big\|_{ L^p(X)} + \big\|(A(w) u_n )_{n=k}^N\big\|_{ L^p(X)} 
        \leq C_{p}(w) \big\|(f_n)_{n=k}^N\big\|_{ L^p(X)},
    \end{equation*}
    where the constant is independent of $N$ and $\tau$. Moreover, the bound is uniform 
    in bounded sets of $W$: for every $R>0$,
\[C_{p}(w) \le C_{p,R} \quad\text{if }\ \|w\|_{W} \le R.\]

\begin{lemma}\label{lem:unif-maxreg-bdf}
For the operators given by $A(w)u=-\nabla\cdot (a(w) \nabla u)$ with homogeneous 
Dirichlet boundary conditions and a smooth positive function $a(\cdot)$ and the BDF 
methods up to order  $k\le 6$, the uniform discrete maximal regularity property {\rm(i')}
 is fulfilled in the situations {\rm (P1)} and {\rm (P2)} of the previous section.
\end{lemma}

The proof of this lemma will be given in Section~\ref{Sec:MLpBDF}. We note that for a 
fixed $w$, such a result was first proved in \cite{KLL} 
for $X=L^q$. The main novelty of Lemma~\ref{lem:unif-maxreg-bdf} 
is thus the case $X=W^{-1,q}$ and the local $W$-uniformity of the result.

For the BDF methods of orders $k=3,4,5$, we need a further condition that complements (v):

(v') For all $v,w\in W$ that are bounded by $R$ in the $W$-norm, and for all $u\in V$,
\[\| (A(v)-A(w))u \|_{V'} \le C_R\, \| v-w \|_W \, \| u \|_V.\]
This condition is trivially satisfied in the situations {\rm (P1)} and {\rm (P2)}.

\subsection{Stability estimate} \label{subsec:stability}
In the following, let $\hat u_n = u_n$ for the fully implicit BDF method and 
$\hat u_n= \sum_{j=0}^{k-1} \gamma_j u_{n-j-1}$ for the linearly implicit BDF method.

Suppose now that $u_n,u_n^\star\in D$ $(n=0,\dotsc,N)$ solve
\begin{equation}
\label{ivp-numer-bdf}
\dot u _n + A(\hat u_n)u_n=f_n, \qquad k\le n \le N,\\
\end{equation}
and the perturbed equation
%
\begin{equation}
\label{ivp2-bdf}
\dot u _n^\star + A(\hat u_n^\star)u_n^\star=f_n+d_n, \qquad k\le n \le N,\\
\end{equation}
respectively, where it is further assumed that 
\begin{equation}\label{u-star-lip}
\| u^\star_m - u^\star_n \|_W \le L (m-n)\tau, \qquad 0 \le n \le m \le N,
\end{equation}
and the defect $(d_n)$ is  bounded by
\begin{equation}
\label{est-delta-bdf}
 \big\|(d_n)_{n=k}^N\big\|_{ L^p(X)} \le \delta 
\end{equation} 
and the errors of the starting values are bounded by
\begin{equation}
\label{est-err-start}
\frac1\tau\,   \big\|( u_i - u_i^\star)_{i=0}^{k-1}\big\|_{ L^p(X)} \le \delta.
\end{equation}
We then have the following time-discrete version of Proposition~\ref{prop:basic}.

\begin{proposition}
\label{prop:stability} Consider time discretization by a fully implicit or linearly implicit BDF 
method of order $k\le 5$. 
In the above setting of the $W$-locally uniform discrete maximal regularity \emph{(i')} 
and \emph{(ii)--(v)} $($and additionally \emph{(v')} if $k=3,4,5\,)$ and 
\eqref{ivp-numer-bdf}--\eqref{est-err-start}, there exist $\delta_0>0$ and $\tau_0>0$ 
such that for $\delta\le\delta_0$ and $\tau\le \tau_0$, 
the errors $e_n=u_n-u_n^\star$ between the solutions of  \eqref{ivp-numer-bdf} and 
\eqref{ivp2-bdf} are bounded by
\begin{align*}
 \big\|(\dot e_n )_{n=k}^N\big\|_{ L^p(X)} +  \big\|(e_n )_{n=k}^N\big\|_{ L^p(D)}
 &\le C\delta
\\
 \big\|(e_n )_{n=k}^N\big\|_{ L^\infty(W)} &\le C\delta,
\end{align*}
where $C$ depends on $\| (u^\star_n)_{n=0}^N \|_{L^\infty(W)}$ and 
$\| (u^\star_n)_{n=0}^N \|_{L^p(D)}$ and on $L$ of \eqref{u-star-lip}, but is independent 
of $\delta$ and of $N$ and $\tau$ with $N\tau\le T$.
\end{proposition}

This stability result will be proved in this and the next section.

\subsection{Auxiliary results by Dahlquist and Nevanlinna \& Odeh}
We will prove Proposition \ref{prop:stability} for the linearly 
implicit scheme similarly to the proof of Proposition~\ref{prop:basic}. To be able to 
use energy estimates in the time-discrete setting of BDF methods, we need the
 following auxiliary results.

\begin{lemma}{\upshape (Dahlquist \cite{D}; see also \cite{BC} and \cite[Section V.6]{HW})}
\label{Le:Dahl}
Let $\delta(\zeta) =\delta_k\zeta^k+\dotsb+\delta_0$ and 
$\mu(\zeta)=  \mu_k\zeta^k+\dotsb+\mu_0$ be polynomials of degree at 
most $k\ ($and at least one of them of degree $k)$
that have no common divisor. 
Let $(\cdot,\cdot)$ be an inner product with associated norm $|\cdot|.$
If
\[\Real \frac {\delta(\zeta)}{\mu(\zeta)}>0\quad\text{for }\, |\zeta|<1,\]
then there exists a positive definite symmetric matrix $G=(g_{ij})\in \R^{k\times k}$ 
and real $\kappa_0,\dotsc,\kappa_k$ such that for $v_0,\dotsc,v_k$ in 
the real inner product space,
\begin{equation*}
\Big (\sum_{i=0}^k\delta_iv_{k-i},\sum_{j=0}^k\mu_jv_{k-j}\Big )=
\sum_{i,j=1}^kg_{ij}(v_{i},v_{j})
-\sum_{i,j=1}^kg_{ij}(v_{i-1},v_{j-1})
+\Big |\sum_{i=0}^k\kappa_iv_{i}\Big |^2.
\end{equation*}
\end{lemma}

In combination with the preceding result for the multiplier $\mu(\zeta)=1-\theta_k\zeta,$ 
the following property of BDF methods up to order $5$ becomes important.

\begin{lemma}{\upshape (Nevanlinna \& Odeh \cite{NO})}\label{Le:NO}
For $k\le 5,$  there exists 
$0\le \theta_k<1$ such that 
for
$\delta (\zeta)= \sum_{ \ell=1}^k \frac 1 \ell  (1-\zeta)^ \ell$, 
\[\Real \frac {\delta(\zeta)}{1-\theta_k\zeta}>0\quad\text{for }\, |\zeta|<1.\]
The smallest possible values of $\theta_k$ are
\[\theta_1=\theta_2=0,\ \theta_3=0.0836,\ \theta_4=0.2878,\ \theta_5=0.8160.
\]
\end{lemma}

Precise expressions for the optimal multipliers for the BDF methods of orders 3, 4, and 5 
are given by Akrivis \& Katsoprinakis \cite{AK}.

An immediate consequence of Lemma \ref{Le:NO} and Lemma \ref{Le:Dahl} is the
relation
\begin{equation}
\label{multiplier}
\Big (\sum_{i=0}^k\delta_iv_{k-i},v_k-\theta_k v_{k-1}\Big )\ge
\sum_{i,j=1}^kg_{ij}(v_{i},v_{j})
-\sum_{i,j=1}^kg_{ij}(v_{i-1},v_{j-1})
\end{equation}
with a positive definite symmetric matrix $G=(g_{ij})\in \R^{k\times k}$;  
it is this inequality  that will play a crucial role in our energy estimates. \medskip

\subsection{Proof of Proposition~\ref{prop:stability} for the linearly implicit BDF methods}
\label{subsec:proof-stability}
We subdivide the proof into four parts (a) to (d) that are analogous to the corresponding 
parts in the proof of Proposition~\ref{prop:basic}. Parts (a)--(c) apply to both the linearly 
and fully implicit BDF methods, whereas the argument in part (d) does not work for the 
fully implicit method.

(a) (\emph{Error equation})
We rewrite the equation for $u_n$ in the form
\[
\dot u_n + A(\hat u^\star _n)u _n=\big (A(\hat u^\star_n)-A(\hat u _n)\big )u _n+f_n
\]
and see that the errors $e_n=u_n-u^\star_n$ satisfy the error equation, for $n\le N$,
\begin{equation}
\label{er-eq-bdf}
\dot e_n + A(\hat u^\star _n)e_n=\big (A(\hat u^\star_n)-A(\hat u_n)\big )u _n-d_n.
\end{equation}
We abbreviate $A_n=A(\hat u^\star _n)$.
For an arbitrary $m \le N$ and for $k\le n\le m$ we further have the error equation 
with a fixed operator
\begin{equation}
\label{er-eq3-bdf}
\dot e _n + A_me _n=\widehat d_n:=(A_m-A_n)e_n
+\big (A(\hat u^\star _n)-A(\hat u _n)\big )u _n-d_n.
\end{equation}
If we redefine $e_0=\dotsb=e_{k-1}=0$, then there appear extra defects for 
$n=k,\dotsc,{2k-1}$, which are bounded by $\delta$ by condition \eqref{est-err-start} 
and are subsumed in $d_n$ in the following.

(b) (\emph{Maximal regularity})
We denote 
\begin{equation}\label{RB-bounds}
 R=\| (u^\star_n)_{n=0}^N \|_{L^\infty(W)}+1\quad\text{and}\quad
B=  \| (u^\star_n)_{n=0}^N \|_{L^p(D)}+1,
\end{equation}
and let $M\le N$ be maximal such that 
\begin{equation} \label{M-RB}
\| (u_n)_{n=0}^{M-1} \|_{L^\infty(W)} \le R \quad\text{and}\quad
\| (u_n)_{n=0}^{M-1} \|_{L^p(D)} \le B.
\end{equation}
By the discrete maximal $L^p$-regularity (i')
we obtain from \eqref{er-eq3-bdf} 
\begin{equation}
\label{max-reg-bdf}
 \big\|(\dot e_n )_{n=k}^m\big\|_{ L^p(X)} +  \big\|(e_n )_{n=k}^m\big\|_{ L^p(D)} \le 
 C_{p,R}\,  \big\|(\widehat d_n )_{n=k}^m\big\|_{ L^p(X)} ,
\end{equation}
and by the bounds (iv) and \eqref{u-star-lip}, we have, for any $\eps>0$,
\begin{align}
\label{d-hat}
\| \widehat d_n \|_{X} & \le 
\|  d_n \|_{X} + C_R L (m  - n)\tau \| e_n \|_{D} 
\\
\nonumber
&\quad +\ \eps\, \| e_n \|_{W} \,  \| u_n \|_{D} + C_{\eps,R}\, \| e_n \|_{H} \,  \| u_n \|_{D}  
\quad\text{for}\,\,\, 
 k\leq n\leq m\leq M .
\end{align}
We take the second term on the left-hand side of \eqref{max-reg-bdf} to power $p$ and denote it by
\[\eta_m = \big\|(e_n )_{n=k}^m\big\|_{ L^p(D)}^p.\]
For the second term on the right-hand side of \eqref{d-hat} we note that, by partial summation,
\begin{align*}
\tau\sum_{n=k}^m (m-n)^p\tau^p \,\| e_n \|_{D}^p 
&= \tau\sum_{n=k}^{m-1} \bigl((m-n)^p-(m-n-1)^p)\tau^{p-1}\,\eta_n 
\\
& \le C \tau\sum_{n=k}^m \bigl((m-n)\tau\bigr)^{p-1} \eta_n.
\end{align*}
Hence we have from \eqref{max-reg-bdf}
\begin{align*}
\eta_m \le  & \ C_{p,R}^p\,  \big\|(\widehat d_n )_{n=k}^m\big\|_{ L^p(X)}^p  
\le C \tau\sum_{n=k}^m \bigl((m-n)\tau\bigr)^{p-1} \eta_n
\\
&+ C\Bigl( \eps B\| (e_n )_{n=k}^m  \|_{L^\infty(W)} + CB \| (e_n )_{n=k}^m\|_{L^\infty(H)}  
+ C \big\|(d_n )_{n=k}^m\big\|_{ L^p(X)} \Bigr)^p.
\end{align*}
With a discrete Gronwall inequality, we therefore obtain from \eqref{max-reg-bdf}
\begin{align}\label{max-reg-aux}
 \big\|(\dot e_n )_{n=k}^M\big\|_{ L^p(X)} +  \big\|(e_n )_{n=k}^M\big\|_{ L^p(D)} \le &\  
 C \eps \| (e_n )_{n=k}^M  \|_{L^\infty(W)} 
 \\ 
 \nonumber
 &\hspace{-2cm}\  + C \| (e_n )_{n=k}^M\|_{L^\infty(H)}  
  + C \big\|(d_n )_{n=k}^M\big\|_{ L^p(X)}.
\end{align}
Next we show that by property (ii) the left-hand side dominates $\| (e_n )_{n=k}^M  \|_{L^\infty(W)}$.
We write $\delta(\zeta)=(1-\zeta)\mu(\zeta)$, where the polynomial $\mu(\zeta)$ of degree 
$k-1$ has no zeros in the closed unit disk, and therefore
\[\frac1{\mu(\zeta)} = \sum_{n=0}^\infty \chi_n \,\zeta^n, \qquad\text{where} \quad
|\chi_n|\le \rho^n \ \text{ with } \rho<1.\]
It follows that
\[\frac{e_n-e_{n-1}}\tau = \sum_{j=0}^n \dot e_{n-j} \, \chi_j\]
and
\[ \Big\|\Bigl(\frac{e_n-e_{n-1}}\tau \Bigr)_{n=k}^M\Big\|_{ L^p(X)} 
\le C  \big\|(\dot e_n )_{n=k}^M\big\|_{ L^p(X)}.\]
If we denote by $e(t)$ the piecewise linear function that interpolates the values $e_n$, 
then we have
\[\| \dot e \|_{L^p(0,M\tau;X)} =  \Big\|\Bigl(\frac{e_n-e_{n-1}}\tau \Bigr)_{n=k}^M \Big\|_{ L^p(X)} 
\quad\text{and}\quad \| e \|_{L^p(0,M\tau;D)} \le C  \big\|(e_n )_{n=k}^M\big\|_{ L^p(D)}.\]
Combining the above inequalities and using property (ii) for $e(t)$, we thus obtain
\[\| (e_n )_{n=k}^M  \|_{L^\infty(W)} \le C \bigl(
 \big\|(\dot e_n )_{n=k}^M\big\|_{ L^p(X)} +  \big\|(e_n )_{n=k}^M\big\|_{ L^p(D)} \bigr).\]
For sufficiently small $\eps$ we can therefore absorb the first term of the right-hand 
side of \eqref{max-reg-aux}  in the left-hand side to obtain
\begin{equation}
\label{max-reg-2-bdf}
 \big\|(\dot e_n )_{n=k}^M\big\|_{ L^p(X)} +  \big\|(e_n )_{n=k}^M\big\|_{ L^p(D)} 
 \le C \| (e_n )_{n=k}^M\|_{L^\infty(H)}  + C \delta.
\end{equation}

(c) (\emph{Energy estimate}) To bound the first term on the right-hand side of \eqref{max-reg-2-bdf} 
we use the energy estimate obtained by testing \eqref{er-eq-bdf} with $e_n-\theta_k e_{n-1}$ 
with $\theta_k$ from Lemma~\ref{Le:NO}: 
\begin{equation}\label{BDF6}
\frac 1\tau\big (e_n-\theta_ke_{n-1},\sum_{j=0}^k\delta_je_{n-j}\big ) 
+ \langle e_n,A_n e_n \rangle - \theta_k\langle e_{n-1},A_n e_n \rangle = 
\langle e_n-\theta_k e_{n-1}, \widetilde d_n \rangle
\end{equation}
with 
\[\widetilde d_n = \big (A(\hat u_n^\star)-A(\hat u_n)\big )u _n -d_n .\]
Now, with the notation $E_n:=(e_n,\dotsc,e_{n-k+1})^T$ 
and the norm $|E_n|_G$ given by
\[|E_n|_G^2=\sum_{i,j=1}^k g_{ij}(e_{n-k+i},e_{n-k+j}),\]
using \eqref{multiplier}, we can estimate the first term on the left-hand
side from below in the form
\begin{equation}
\label{BDF3}
\big (e_n-\theta_ke_{n-1},\sum_{j=0}^k\delta_je_{n-j}\big )
\ge |E_n|_G^2-|E_{n-1}|_G^2.
\end{equation}
In the following 
we denote by $\|\cdot\|_n$ the norm given by $\| v \|_n^2 
= \langle v,A_n v\rangle$, which by (iii) is equivalent to the $V$-norm. Furthermore, 
\[
\langle e_n-\theta_ke_{n-1},A_n e_n\rangle =\|e_n\|_n^2-\theta_k\langle e_{n-1},A_ne_n \rangle,
\]
whence, obviously,
\begin{equation}
\label{BDF4}
\langle e_n-\theta_ke_{n-1},A_ne_n \rangle \ge  \big (1-\frac {\theta_k}2\big )\|e_n\|_n^2
-\frac {\theta_k}2\|e_{n-1}\|_n^2.
\end{equation}
Moreover,
\begin{equation}
\label{BDF5}
\langle e_n-\theta_ke_{n-1},\widetilde  d_n \rangle \le \varepsilon (\|e_n\|_n^2
+\theta_k^2 \|e_{n-1}\|_{n-1}^2)
+C_\varepsilon \|\widetilde d_n\|_{V'}^2,
\end{equation}
for any positive $\varepsilon$.
At this point, we need to relate $\|e_{n-1}\|_n$ back to $\|e_{n-1}\|_{n-1}.$
We have, by the bounds (v') and  \eqref{u-star-lip},
\[\|v\|_n^2-\|v\|_{n-1}^2=\langle v,A_n v\rangle -\langle v,A_{n-1} v\rangle
=\langle v,(A_n-A_{n-1})v\rangle \le C\tau \| v \|_V^2,\]
so that
\begin{equation}
\label{BDF6n}
\|e_{n-1}\|_n^2\le (1+C\tau) \|e_{n-1}\|_{n-1}^2.
\end{equation}
Summing in \eqref{BDF6} from $n=k$ to $m\le M,$ we obtain
\[|E_m|_G^2+\rho \tau \sum_{n=k}^m\|e_n\|_n^2\le
C_\varepsilon \tau \sum_{n=k}^m \|\widetilde d_n\|_{V'}^2,\]
with $\rho:=1-\theta_k-(1+\theta_k)\varepsilon>0$. 

To estimate $\widetilde d_n$, we note that the  bound of (v) yields 
\begin{align*}
\|(A(\hat u^\star_n )-A(\hat u_n ))u_n \|_{V'} &\le \sum_{j=0}^k \Bigl( \eps \| e_{n-j} \|_V 
+ C_{\eps,R} \| e_{n-j} \|_{H} \Bigr) \|u_n\|_{D}
\\
&\le
\sum_{j=0}^k \Bigl( 2\eps \| e_{n-j} \|_{V}^2 + C  \| e_{n-j} \|_{H}^2 \Bigr)  \|u_n\|_{D}.
\end{align*}
Absorbing the terms with $\| e_{n-j} \|_{V}^2$ we therefore obtain, for $k\le m \le M$,
\[ \| e_m \|_H^2 + \tau\sum_{n=k}^m \| e_n \|_V^2 \le C \tau\sum_{n=k}^m \| e_n \|_{H}^2 + 
 C\tau \sum_{n=k}^m \|d_n\|_{V'}^2,\]
 and the discrete Gronwall inequality then yields
\[\|e_n\|_{H} \le C\delta, \qquad k\le n \le M.\]

(d) (\emph{Complete time interval}) Inserting the previous bound in \eqref{max-reg-2-bdf}, we obtain
\[ \big\|(\dot e_n )_{n=k}^M\big\|_{ L^p(X)} +  \big\|(e_n )_{n=k}^M\big\|_{ L^p(D)} \le  C \delta,\]
which by (ii) and the argument in part (b) above further implies
\[\big\|(e_n )_{n=k}^M\big\|_{ L^\infty(W)} \le C \delta.\]
For the linearly implicit BDF method, this implies that 
$\hat u_{M+1}=\sum_{j=0}^{k-1} \gamma_j u_{M-j}$ is bounded by $\|\hat u_{M+1}\|_W \le CR$,
and hence the above arguments can be repeated to yield that for sufficiently small $\delta$ 
the bounds \eqref{M-RB} are  also satisfied for $M+1$, which contradicts the maximality 
of $M$ unless $M=N$.
\qed\medskip

While parts (a)--(c) of the above proof apply also to the fully implicit BDF methods, 
the argument in part (d) does not work for the fully implicit method. Here we need some 
a priori estimate from the existence proof, which is established in the next section.

\section{Existence of numerical solutions for the fully implicit scheme} 
\label{Sec:existence}

While existence and uniqueness of the numerical solution are obvious for the linearly implicit 
BDF method \eqref{bdf:lin-im1}, this is not so for the fully implicit method.
In this section, we prove existence and uniqueness of the numerical solution for the fully implicit 
BDF method \eqref{bdf:fully-im1} and complete the proof of Proposition~\ref{prop:stability} 
for this method.

\subsection{Schaefer's fixed point theorem}
Existence of the solution of the fully implicit BDF method \eqref{bdf:fully-im1} 
is proved with the following result.

\begin{lemma}[Schaefer's fixed point theorem {\cite[Chapter 9.2, Theorem 4]{Evans}}]
\label{THMSchaefer}
Let $W$ be a Banach space and let ${\mathcal M}:W\rightarrow W$ be 
a continuous and compact map. If the set 
\begin{equation}
\bigl\{\phi\in W:\; \phi=\theta {\mathcal M}(\phi) \ \hbox{ for some }\, \theta\in [0,1] \bigr\}
\end{equation}
is bounded in $W$, then the map ${\mathcal M}$ has a fixed point.
\end{lemma}

\subsection{Proof of the existence of the numerical solution and of Proposition~\ref{prop:stability}
for fully implicit BDF methods}
In the situation of Proposition~\ref{prop:stability}, we assume that $u_n\in D, n=k,\dotsc,M-1$, 
are solutions of 
\eqref{bdf:fully-im1} and satisfy 
\begin{equation}\label{mathind}
\|(u_n)_{n=0}^{M-1}\|_{L^\infty(W)}\leq R ,
\end{equation}
with $R$ of \eqref{RB-bounds},
and prove the existence a numerical solution $u_M$ 
for \eqref{bdf:fully-im1}, which also satisfies 
$\|u_M\|_{W}\leq R$.  



We define a map ${\mathcal M}: W \rightarrow W$ in the following way.   
For any $\phi\in W$  we define 
\begin{equation}\label{rho_phi}
\rho_\phi  :=\min\bigg(\frac{\sqrt{\delta} }{\|\phi \|_{W}},1\bigg) .
\end{equation}
Clearly, $\rho_\phi $ depends continuously on $\phi \in W$, and 
\begin{equation}\label{rho_phi-est}
\|\rho_\phi \phi\|_{W}\leq \sqrt{\delta} . 
\end{equation}
Then we define $e_M={\mathcal M}(\phi)$ as the solution of the linear equation
\begin{equation}
\label{BDF:MapM}
\begin{aligned}
\frac 1\tau \sum_{j=0}^k\delta_je_{M-j}
&=A(u_M^\star)e_M +(A(u_M^\star+\rho_\phi\phi)-A(u_M^\star))e_M\\
&\quad +(A(u_M^\star+\rho_\phi\phi )-A(u_M^\star))u_M^\star -d_M.
\end{aligned}
\end{equation}

Using the compact imbedding of $D$ in $W$, the fact that the resolvent operator 
$(\delta_0/\tau+A(u_M^\star+\rho_\phi\phi))^{-1}$ maps from $X$ to $D$, and condition (iv) 
(with $\eps=1$), it is shown that the map ${\mathcal M}$ is continuous and compact. 
Moreover, if the map ${\mathcal M}$ has a fixed point $\phi$ with $\rho_\phi=1$, then 
$e_M={\mathcal M}(\phi)$ is a solution of 
\begin{equation}
\label{BDF:MapM00}
\begin{aligned}
\frac 1\tau \sum_{j=0}^k\delta_je_{M-j}
&=A(u_M^\star)e_M +(A(u_M^\star+e_M)-A(u_M^\star))e_M \\
&\quad 
+(A(u_M^\star+e_M)-A(u_M^\star))u_M^\star -d_M ,
\end{aligned}
\end{equation}
and $u_M:=u_M^\star+e_M$ is the solution of \eqref{bdf:fully-im1} with $n=M$.

To apply Schaefer's fixed point theorem, we assume that $\phi \in W$ 
and $\phi=\theta{\mathcal M}(\phi)$ for some $\theta\in[0,1]$.  
To prove existence of a fixed point for the map ${\mathcal M}$, we only need to prove 
that all such $\|\phi\|_{W}$ are uniformly bounded.

Let $e_M={\mathcal M}(\phi)$.  
Then $\phi=\theta e_M$ and \eqref{BDF:MapM} implies that $e_M$ is the solution of 
\begin{equation}
\label{BDF:FixedP00}
\begin{aligned}
\frac 1\tau \sum_{j=0}^k\delta_je_{M-j}
&=A(u_M^\star)e_M +(A(u_M^\star+\theta \rho_{\phi} e_M)-A(u_M^\star))e_M\\
&\quad 
+(A(u_M^\star+\theta \rho_\phi e_M)-A(u_M^\star))u_M^\star -d_M 
\end{aligned}
\end{equation}
and 
\begin{equation}
\label{BDF:FixedP}
\begin{aligned}
\frac 1\tau \sum_{j=0}^k\delta_je_{n-j}
&=A(u_n^\star)e_n +(A(u_n^\star+ e_n)-A(u_n^\star))e_n\\
&\quad 
+(A(u_n^\star+ e_n)-A(u_n^\star))u_n^\star -d_n 
\end{aligned}
\end{equation}
for $n=k,\dotsc, M-1$, satisfying
\begin{equation}\label{rhothetan}
\|\theta \rho_\phi e_M\|_{W}\leq \sqrt{\delta} . 
\end{equation}
In the same way as in Section~\ref{subsec:proof-stability},
we obtain
\begin{equation}
\label{BDF:fully-im-ErrEst}
\|(\dot e_n)_{n=k}^M\|_{L^p(X)}  +\|(e_n)_{n=k}^M\|_{L^p(D)}\leq  C\delta 
\end{equation}
and
\begin{equation}
\label{BDF:fully-im-W1inftyErr}
\|(e_n)_{n=0}^M\|_{L^\infty(W)} \leq C\delta.  
\end{equation}
Since $\phi =\theta e_M$ with $\theta\in[0,1]$, the last inequality implies uniform 
boundedness of $\|\phi\|_{W}$ with respect to $\theta\in[0,1]$, 
and this implies the existence of a fixed point $\phi$ for the map ${\mathcal M}$ 
by Lemma \ref{THMSchaefer}.  Moreover, in view of \eqref{BDF:fully-im-W1inftyErr} 
and since the fixed point $\phi$ satisfies $\phi=e_M$, 
for sufficiently small $\delta $ we have
\begin{equation}
\|\phi\|_{W} \leq \sqrt{\delta  } ,
\quad\text{and so}\quad \rho_\phi =1.
\end{equation}
This proves the existence of a solution of \eqref{bdf:fully-im1} for sufficiently small $\delta$, 
and the solution satisfies \eqref{BDF:fully-im-ErrEst} and \eqref{BDF:fully-im-W1inftyErr} and 
hence also $\|u_M\|_W \le R$, which completes the proof of Proposition~\ref{prop:stability} 
for the fully implicit BDF methods. 
\qed

\subsection{Uniqueness of the numerical solution} Suppose that there are two numerical solutions 
$u_n,\widetilde u_n\in D$  of \eqref{bdf:fully-im1}, both with $W$-norm bounded by $R$ as in 
the proof of existence. By induction, we assume unique numerical solutions 
$u_j$ with $W$-norm bounded by $R$ for $j<n$. The difference $e_n=u_n-\widetilde u_n$ 
then satisfies the equation
\[\frac {\delta_0}\tau \, e_n + A(\widetilde u_n) e_n = \bigl( A(u_n)-A(\widetilde u_n)\bigr) u_n.\]
We test this equation with $e_n$ and note that, with some $\alpha_R>0$ depending on $R$, 
we have, using conditions (iii) and (v),
\[\frac {\delta_0}\tau \, \| e_n\|_H^2 + \alpha_R \| e_n \|_V^2 \le \|e_n \|_V \bigl( \eps \|e_n\|_V 
+ C_{\eps,R} \|e_n\|_H \bigr) \|u_n \|_D,\]
which implies that there exists $\tau_R>0$ such that for $\tau\le \tau_R$, we have $e_n=0$.
We have thus shown uniqueness of the numerical solution $u_n\in D$ with $\|u_n\|_W\le R$ 
for stepsizes $\tau\le \tau_R$. 

\section{Consistency error} \label{Sec:consistency}
The order of both the $k$-step fully implicit BDF method, described by the coefficients
$\delta_0,\dotsc,\delta_k$ and $1,$ and of the explicit $k$-step BDF method,
that is the method  described by the coefficients
$\delta_0,\dotsc,\delta_k$ and $\gamma_0,\dotsc,\gamma_{k-1},$ is $k,$ i.e.,
\begin{equation}
\label{order}
\sum_{i=0}^k(k-i)^\ell \delta_i=\ell k^{\ell-1}=
\ell \sum_{i=0}^{k-1}(k-i-1)^{\ell-1} \gamma_i,\quad \ell=0,1,\dotsc,k.  
\end{equation}
The defects (consistency errors) $d^n$ and $\tilde d^n$ of the schemes 
\eqref{bdf:fully-im1} and  \eqref{bdf:lin-im1} for the solution $u$ of \eqref{ivp}, 
i.e., the amounts by which the exact solution misses satisfying \eqref{bdf:fully-im1} 
and \eqref{bdf:lin-im1}, respectively, are given by 
\begin{equation}
\label{cons1}
d_n=\frac1 \tau\sum\limits^k_{j=0}\delta_ju(t_{n-j}) + A(u(t_n))u(t_n),
\end{equation}
and
\begin{equation}
\label{cons2}
\tilde d_n=\frac 1\tau \sum_{j=0}^k\delta_ju(t_{n-j})+
A\Big (\sum\limits^{k-1}_{j=0}\gamma_ju(t_{n-j-1})\Big )u(t_n),
\end{equation}
$n=k,\dotsc,N,$ respectively.
 
\begin{lemma} \label{lem:defect}
Under the regularity requirements \eqref{reg-1} or \eqref{reg-2}, 
the defects $d_n$ and $\tilde d_n$ are bounded by
\begin{equation}
\label{cons-err-est-1}
\max_{k\le n\le N}\|d_n\|_{W^{-1,q}(\varOmega)} \le C\tau^k, \quad 
\max_{k\le n\le N}\|\tilde d_n\|_{W^{-1,q}(\varOmega)} \le C\tau^k
\end{equation}
in case of \eqref{reg-1}, and by
\begin{equation}
\label{cons-err-est-2}
\max_{k\le n\le N}\|d_n\|_{L^q(\varOmega)} \le C\tau^k, \quad 
\max_{k\le n\le N}\|\tilde d_n\|_{L^q(\varOmega)} \le C\tau^k
\end{equation}
in case of \eqref{reg-2}.
\end{lemma}


\begin{proof} Since the proofs for \eqref{cons-err-est-1} and \eqref{cons-err-est-2} are 
almost identical, we just present the proof of \eqref{cons-err-est-2}.
We first focus on the implicit scheme \eqref{bdf:fully-im1}. 
Using the differential equation 
in \eqref{ivp}, we rewrite \eqref{cons1} in the form
\begin{equation}
\label{cons4}
d_n=\frac1\tau\sum\limits^k_{j=0}\delta_ju(t_{n-j}) -  \partial_t u(t_n).
\end{equation}
By Taylor expansion about $t_{n-k},$ we see that,
due to the order conditions of the implicit BDF method,
i.e., the first equality in \eqref{order}, leading  terms of order up to $k-1$ 
cancel, and we obtain 
%
%
%
\begin{equation}
\label{cons5}
 d_n=\frac 1{k!}\Bigg [ \frac1\tau\sum\limits^k_{j=0} \delta_j\!
\int_{t_{n-k}}^{t_{n-j}}(t_{n-j}-s)^ku^{(k+1)}(s)\, ds
-k \int_{t_{n-k}}^{t_n}(t_n-s)^{k-1}u^{(k+1)}(s)\, ds\Bigg ];
\end{equation}
here, we used the notation $u^{(m)}:=\frac {\partial^m u}{\partial t^m}.$
Taking the $L^q$ norm on both sides 
of \eqref{cons5}, 
we obtain the desired optimal order consistency estimate  \eqref{cons-err-est-2}
for the scheme \eqref{bdf:fully-im1}.

Next, concerning the scheme \eqref{bdf:lin-im1}, from \eqref{cons1} and \eqref{cons2} 
we immediately obtain the following relation between $\tilde d_n$ and $d_n$
\begin{equation}
\label{cons9}
\tilde d_n=d_n+\big ( A(u(t_n))-A (\hat u(t_n) )\big )u(t_n)
\end{equation}
with
\begin{equation}
\label{bdf:hat-sol}
\hat u(t_n):=\sum\limits^{k-1}_{i=0}\gamma_iu(t_{n-i-1}).
\end{equation}
%
By Taylor expanding about $t_{n-k},$ 
leading  terms of order up to $k-1$ cancel again, this time due to the second
equality in \eqref{order}, and we obtain
\begin{equation*}
u(t_n)- \hat u(t_n)=\frac 1{(k-1)!}\Bigg [\! \int_{t_{n-k}}^{t_n}\!\!(t_n-s)^{k-1}u^{(k)}(s) ds
-\!\sum\limits^k_{j=0}\gamma_j \!\int_{t_{n-k}}^{t_{n-j-1}}\!\!(t_{n-j-1}-s)^{k-1}u^{(k)}(s) ds\!\Bigg ]\!,
\end{equation*}
whence, taking the $W^{1,q}$ norm on both sides 
of this relation, we immediately infer that 
%
\begin{equation}
\label{cons10}
\|u(t_n)- \hat u(t_n)\|_{W^{1,q}(\varOmega)}\le C\tau^k.
\end{equation}

Now, $\big ( A(u(t_n))-A (\hat u(t_n) )\big )u(t_n)=
\nabla \cdot \big (\big ( a(u(t_n))-a (\hat u(t_n) )\big )\nabla u(t_n)\big ),$
whence
%
\begin{align*}
&\|\big ( A(u(t_n))-A (\hat u(t_n) )\big )u(t_n)\|_{L^q(\varOmega)}\\
&=\|\nabla\cdot\big ( (a(u(t_n))-a (\hat u(t_n) )\nabla u(t_n)\big )\|_{L^q(\varOmega)}\\
&{}\le C\| a(u(t_n))-a (\hat u(t_n) )\|_{L^\infty(\varOmega)}
\|u(t_n)\|_{W^{2,q}(\varOmega)} \\
&\quad 
+C\| a(u(t_n))-a (\hat u(t_n) )\|_{W^{1,q}(\varOmega)}
\|u(t_n)\|_{W^{1,\infty}(\varOmega)} \\
&\le C\| u(t_n)-\hat u(t_n)\|_{W^{1,q}(\varOmega)} ;
\end{align*}
%
therefore, in view of \eqref{cons10}, we have 
\begin{equation}
\label{cons11}
\|\big ( A(u(t_n))-A (\hat u(t_n) )\big )u(t_n)\|_{L^q(\varOmega)}\le C\tau^k.
\end{equation}
Combining \eqref{cons11} and the bound for $d_n$, 
we obtain the desired optimal order consistency estimate  \eqref{cons-err-est-2}
also for the scheme \eqref{bdf:lin-im1}. 
\end{proof}

\section{Proof of Theorems~\ref{Th:main-1} and~\ref{Th:main-2}} 
\label{Sec:proof}
The cases (P1) and (P2) of Section~\ref{subsec:abstract-framework} correspond to the 
situation in Theorems~\ref{Th:main-1} and \ref{Th:main-2}, respectively. 
Lemma~\ref{lem:framework} ensures that the considered problems are of the type considered 
in the abstract framework of Section~\ref{subsec:abstract-framework}, and 
Lemma~\ref{lem:unif-maxreg-bdf} ensures the required uniform discrete maximal regularity 
of the BDF methods. Lemma~\ref{lem:defect} yields that the situation of 
Section~\ref{subsec:stability} holds with $u_n^\star=u(t_n)$ and $\delta\le C\tau^k$. 
The error bounds of Theorems~\ref{Th:main-1} and~\ref{Th:main-2} then follow from 
Proposition~\ref{prop:stability}.

It remains to give the proofs of Lemmas \ref{lem:framework} and~\ref{lem:unif-maxreg-bdf}. 
This is done in the final two sections.

\section{$W$-locally uniform maximal regularity}
\label{Sec:unif-maxreg}

\subsection{Proof of (i) in Lemma \ref{lem:framework}} 

Let $\varOmega\subset\R^d$ be a bounded Lipschitz domain 
and consider the following initial and boundary value problem for a linear
parabolic equation, with a time-independent self-adjoint operator,
\begin{equation}\label{PDECauchy}
\left\{
\begin{alignedat}{2}
&\partial_tu(x,t)-\nabla\cdot \big(b(x)\nabla u(x,t)\big) =0\ \ 
&&\text{for}\,\,(x,t)\in \varOmega\times \R_+,\\[2pt]
&u(x,t)=0 &&\text{for}\,\,(t,x)\in \partial\varOmega\times \R_+ ,\\[2pt]
&u(x,0)=u_0(x) &&\text{for}\,\,x\in \varOmega ,
\end{alignedat}\right.
\end{equation}
where the coefficient $b(x)$ satisfies
\begin{equation}\label{ellipt} 
K_0^{-1}\leq b (x) \leq K_0.  
\end{equation} 
We consider $W=C^\alpha(\overline\varOmega)$ and $W=C^{1,\alpha}(\overline\varOmega)$ 
in the settings (P1) and (P2), respectively.
In this section, we combine results from the literature and prove $W$-locally uniform 
maximal parabolic regularity of \eqref{PDECauchy}, 
where the constant depends only on $K_0$ 
and $\|b\|_W$.

Let $\{E_2(t)\}_{t>0}$ denote the semigroup of operators on 
$L^2(\varOmega)$, 
which maps $u_0$ to $u(t,\cdot)$, given by \eqref{PDECauchy} 
and let $A_2$ denote the generator of this semigroup. 
Then $\{E_2(t)\}_{t>0}$ extends to a bounded analytic semigroup 
on $L^2(\varOmega)$, in the sector $\varSigma_{\theta}=\{z\in\C\,:\,z\ne 0, |\arg z|< \theta\}$, 
where $\theta$ can be arbitrarily close to $\pi/2$ 
(see \cite{Davis89,Ouhabaz}), 
and the kernel $G(t,x,y)$ of the semigroup $\{E_2(t)\}_{t>0}$ 
has an analytic extension to the right-half plane, satisfying 
(see \cite[p.\ 103]{Davis89}) 
\begin{equation}
|G(z,x,y)|\leq C_\theta|z|^{-\frac{d}{2}}
\e^{-\frac{|x-y|^2}{C_\theta |z|}}, 
\quad \forall\, z\in \varSigma_{\theta},\,\,\forall\, x,y\in\varOmega ,
\quad\forall\,\theta\in(0,\pi/2), 
\label{GKernelE0}
\end{equation}
where the constant $C_\theta$ 
depends only on $K_0$ and $\theta$. In other words, 
the operator $A=-e^{i\theta}A_2$ satisfies the condition 
of \cite[Theorem 8.6]{KW}, with $m=2$ and 
$g(s)=C_\theta e^{-s^2/C_\theta}$ (also see \cite[Remark 8.23]{KW}).
As a consequence of \cite[Theorem 8.6]{KW}, 
$E_2(t)$ extends to an analytic 
semigroup $E_q(t)$ on $L^q(\varOmega)$, $1<q<\infty$, which is $R$-bounded 
in the sector $\varSigma_\theta$ 
for all $\theta\in(0,\pi/2)$ 
and the $R$-bound depends only on $C_\theta$ and $q$. (We refer to \cite{KW} for a 
discussion of the notion of $R$-boundedness.)
To summarize, we have the following lemma. 

\begin{lemma}[Angle of $R$-boundedness of the semigroup]\label{RbdSg}
For any given $1<q<\infty$, the semigroup 
$\{E_q(t):L^q(\varOmega)\rightarrow L^q(\varOmega)\}_{t>0}$ 
defined by the parabolic problem 
\eqref{PDECauchy} is $R$-bounded in the sector 
$\varSigma_\theta=\{z\in\C: |{\rm arg}(z)|<\theta\}$ 
for all $\theta\in(0,\pi/2)$,  
and the $R$-bound depends only on $K_0$, $\theta$ and $q$. 
\end{lemma}

According to  Weis' characterization of maximal $L^p$-regularity \cite[Theorem 4.2]{Weis2}, 
Lemma \ref{RbdSg} implies the following two results.

\begin{lemma}[Angle of $R$-boundedness of the resolvent]\label{RbdRes}
Let $A_q$ be the generator of the semigroup $\{E_q(t)\}_{t>0}$ 
defined by the parabolic problem 
\eqref{PDECauchy}, where $1<q<\infty$. 
Then the set $\{ \lambda(\lambda-A_q)^{-1}\,:\, \lambda\in \varSigma_\theta \}$
 is $R$-bounded in the sector 
$\varSigma_\theta=\{z\in\C: |{\rm arg}(z)|<\theta\}$ 
for all $\theta\in(0,\pi)$,  
and the $R$-bound depends only on $K_0$, $\theta$ and $q$. 
\end{lemma}

\begin{lemma}[Maximal $L^p$-regularity]\label{LemmaMaxLp}
Let $A_q$ be the generator of the semigroup $\{E_q(t)\}_{t>0}$ 
defined by the parabolic problem 
\eqref{PDECauchy}. 
Then the solution $u(t)$ of the parabolic initial value problem
\begin{equation} 
\label{IVPq} 
    \left\{
    \begin{aligned}
        u'(t)&{} = A_qu(t) + f(t), \quad t>0,\\
        u(0) &{}= 0,
    \end{aligned}
    \right.
\end{equation}
%
belongs to $D(A_q)$ for almost all $t\in\R_+$, and 
\begin{equation}\label{MaxLpAbstr}
\|u'\|_{L^p(\R_+;L^q(\varOmega))}
+\|A_qu\|_{L^p(\R_+;L^q(\varOmega))}\leq C_{p,q}
\|f\|_{L^p(\R_+;L^q(\varOmega))}, 
\end{equation} 
for all $f\in L^p(\R_+;L^q(\varOmega))$ and $1<p,q<\infty$, 
where $C_{p,q}$ depends only on $p, q$ and~$K_0$. 
\end{lemma}

If the domain $\varOmega$ is smooth and 
$b\in C^{1,\alpha}(\bar\varOmega)$, then (see \cite[Chapter 3, Theorems 6.3--6.4]{CW98}) 
\[\|u\|_{W^{2,q}(\varOmega)}
\leq C_{q} \|A_qu\|_{L^q(\varOmega)} ,\] 
where the constant $C_q$ depends only on 
$K_0,q,\alpha$ and $\|b\|_{C^{1,\alpha}(\overline\varOmega)}$. 
Hence, Lemma \ref{LemmaMaxLp} implies (i) for the case (P2).

In the case (P1), we need to use the following result. 

\begin{lemma}\label{THMJK}
For any given bounded Lipschitz domain $\varOmega\subset\R^d, d=2,3$, 
the solution of the elliptic boundary value problem 
\begin{equation} 
\label{EllipticEq}
\left\{
\begin{alignedat}{2}
&\nabla\cdot(b(x)\nabla u) = f\quad&&\text{in}\,\,\,\varOmega, \\
&u=0 &&\text{on}\,\,\,\partial\varOmega,
\end{alignedat}\right.
\end{equation} 
satisfies
\begin{equation}\label{W1quf}
\|u\|_{W^{1,q}(\varOmega)}\leq C_q\|f\|_{W^{-1,q}(\varOmega)} ,
\quad\forall\, q_d'<q<q_d ,
\end{equation} 
where $q_d>2$ is some 
constant which depends on the domain and 
$1/q_d+1/q_d'=1$, and the constant $C_q$ depends on $q$, $K_0$, $\varOmega$, $\alpha$ 
and $\|b\|_{C^\alpha(\overline\varOmega)}$. 
In particular, $q_2>4$ and $q_3>3$ 
for any given bounded Lipschitz domain $\varOmega\subset\R^d$, 
$d=2,3$. 
\end{lemma}
Lemma \ref{THMJK} was proved in \cite[Theorem 0.5]{JK95} 
for constant coefficient $b(x)$, which can be extended to  
variable coefficient $b(x)$ by a standard perturbation 
argument. 
By using Lemmas \ref{LemmaMaxLp} and \ref{THMJK}, we can also prove 
the following maximal $L^p$-regularity on $W^{-1,q}(\varOmega)$, 
which implies (i) for the case (P1) 
(with $X=W^{-1,q}(\varOmega)$ and $d<q<q_d$). 
\begin{lemma}\label{Le:max-reg-con0}
If the domain $\varOmega$ is Lipschitz continuous 
and $b \in C^\alpha(\overline\varOmega)$ 
for some $\alpha\in(0,1)$, then 
the solution of \eqref{IVPq} satisfies 
\begin{equation}
\|\partial_tu\|_{L^p(\R_+;W^{-1,q}(\varOmega))}
+\|u\|_{L^p(\R_+;W^{1,q}(\varOmega))}
\leq C_{p,q}\|f \|_{L^p(\R_+;W^{-1,q}(\varOmega))},  
\label{MaxLpW1q0}
\end{equation}
for all $f\in L^p(\R_+;W^{-1,q}(\varOmega))$, $1<p<\infty$  and $q_d'<q<q_d$;  
the constant $C_{p,q}$ depends only on $p$, $q$, $K_0$,  $\varOmega$, $\alpha$ and 
$\|b\|_{C^\alpha(\overline\varOmega)}$. 
\end{lemma}
\begin{proof}
Since the solution of \eqref{IVPq} is given by 
\[u(t)=\int_0^t E_q(t-s)f(s)\d s ,\]
Lemma \ref{LemmaMaxLp} implies
that the map from $f$ to $A_q u$ given by the
formula 
\[A_qu(t)=\int_0^t A_qE_q(t-s)f(s)\d s\]
is bounded in $L^p(0,T;L^q(\varOmega))$. In other words, 
if we define 
\[w(t):=-\int_0^t  A_qE_q(t-s)(-A_q)^{-1/2}f(s)\d s ,\]
then we have 
\begin{equation}\label{w-f}
\|w\|_{L^p(0,T;L^q(\varOmega))} 
\leq C_{p,q}\|(-A_q)^{-1/2}f\|_{L^p(0,T;L^q(\varOmega))} ,
\end{equation}
where the fractional power operator $(-A_q)^{-1/2}$ is well 
defined (due to the self-adjointness and positivity of $-A_q$)
and commutes with $A_q$. It is straightforward to check that 
\[\nabla u=\nabla (-A_q)^{-1/2}w  .\]
Since the Riesz transform $\nabla (-A_q)^{-1/2}$ is bounded on $L^q(\varOmega)$ for 
$q_d'<q<q_d$ (see Appendix), it follows that $\|\nabla u\|_{L^p(0,T;L^{q}(\varOmega))}
\leq C_{p,q}\|w\|_{L^p(0,T;L^{q}(\varOmega))}$; therefore, in view of \eqref{w-f}, we have
\begin{equation}\label{nablaun0}
\|\nabla u\|_{L^p(0,T;L^{q}(\varOmega))}
\leq C_{p,q}\|(-A_q)^{-1/2}f\|_{L^p(0,T;L^q(\varOmega))} .
\end{equation}
Moreover, since $(-A_q)^{-1/2}\nabla\cdot\,$ is the dual of the Riesz transform 
$\nabla (-A_q)^{-1/2}$, 
it is also bounded on $L^q(\varOmega)$ for any $q_d'<q<q_d$, 
and so we have
\begin{align*}
\|(-A_q)^{-1/2}f\|_{L^p(0,T;L^q(\varOmega))}
&{}
=\|(-A_q)^{-1/2}\nabla\cdot \nabla\varDelta^{-1} f\|_{L^p(0,T;L^q(\varOmega))}\\ 
&{}\leq \|\nabla\varDelta^{-1} f\|_{L^p(0,T;L^q(\varOmega))},
\end{align*}
whence
\begin{equation}\label{Aq-12fn0}
\|(-A_q)^{-1/2}f\|_{L^p(0,T;L^q(\varOmega))}
\leq C_{p,q}\|f\|_{L^p(0,T;W^{-1,q}(\varOmega))}.
\end{equation}
\eqref{nablaun0}--\eqref{Aq-12fn0} yield \eqref{MaxLpW1q0}.  
\end{proof}

\subsection{Proof of Lemma~\ref{lem:unif-maxreg-bdf}}\label{Sec:MLpBDF}

In this section, we consider the BDF time discretization of \eqref{IVPq}: 
\begin{align}
&\label{def:BDF} 
\frac{1}{\tau} \sum_{j=0}^k \delta_j u_{n-j}   
=A_qu_n+f_n, \qquad n \geq k, \\
&\label{starting-from-zero}
u_0=0\quad\text{and}\quad 
u_1,\dotsc,u_{k-1}\,\,\,\text{given (possibly nonzero)}.
\end{align}

In view of Lemma \ref{RbdRes}, we have the following result, 
which implies Lemma \ref{lem:unif-maxreg-bdf} for the case (P2). 

\begin{proposition}[{\cite[Theorems 4.1--4.2 and Remark 4.3]{KLL}}]\label{Pr:BDFk}
For $1\leq k\leq 6$, the solution of \eqref{def:BDF}--\eqref{starting-from-zero} 
satisfies
\begin{equation}
\begin{aligned}
        &\big\|(\dot u_n )_{n=k}^N\big\|_{L^p(L^q(\varOmega))} 
        + \big\|(A_qu_n )_{n=k}^N\big\|_{L^p(L^q(\varOmega))}\\
        &\leq C_{p,q}\bigg (\tau \sum_{n=1}^{k-1} 
        \Big\|\frac{u_n-u_{n-1}}{\tau}\Big\|^p_{L^q(\varOmega)}\bigg )^{\frac{1}{p}} 
+C_{p,q}\bigg (\tau \sum_{n=1}^{k-1} \|A_qu_n\|^p_{L^q(\varOmega)}\bigg )^{\frac{1}{p}} \\
&\quad + C_{p,q}\big\|(f_n)_{n=k}^N\big\|_{L^p(L^q(\varOmega))} ,
\end{aligned}
\end{equation}
 for $1<p,q<\infty$,   where the constant 
 $C_{p,q}$ depends only on $K_0$ and $q,$ i.e.,  
it is independent of $\tau, N$ and $b$. 
\end{proposition}

By applying Proposition \ref{Pr:BDFk}, 
we prove the following result, 
which implies Lemma \ref{lem:unif-maxreg-bdf} for the case (P1) 
(with $X=W^{-1,q}(\varOmega)$ 
and $d<q<q_d$).  

\begin{proposition}\label{Th:BDFk}
Let $1\leq k\leq 6$. 
If the domain $\varOmega$ is Lipschitz continuous 
and the coefficient satisfies $b \in C^\alpha(\overline\varOmega)$ 
for some $\alpha\in(0,1)$, then 
the solution of \eqref{def:BDF}--\eqref{starting-from-zero} satisfies 
\begin{equation}\label{BDFconvex}
\begin{aligned}
        &\big\|(\dot u_n )_{n=k}^N\big\|_{L^p(W^{-1,q}(\varOmega))} 
        + \big\|(u_n )_{n=k}^N\big\|_{L^p(W^{1,q}(\varOmega))}\\
        &\leq C_{p,q}\bigg (\tau \sum_{n=1}^{k-1} 
        \Big\|\frac{u_n-u_{n-1}}{\tau}\Big\|^p_{W^{-1,q}(\varOmega)}\bigg )^{\frac{1}{p}} 
+C_{p,q}\bigg (\tau \sum_{n=1}^{k-1} \|u_n\|^p_{W^{1,q}(\varOmega)}\bigg )^{\frac{1}{p}}\\
&\quad + C_{p,q}\big\|(f_n)_{n=k}^N\big\|_{L^p(W^{-1,q}(\varOmega))} ,
\end{aligned}
\end{equation}
for all $1<p<\infty$ and $q_d'<q<q_d$, and the constant 
 $C_{p,q}$ depends only on $K_0, q$, $\varOmega$, $\alpha$ and 
$\|b\|_{C^\alpha(\overline\varOmega)}$
$($independent of $\tau$ and $N).$ 
  \end{proposition}

\begin{proof} 
In view of 
\cite[Remark 4.3]{KLL}, we only need to consider the case 
$u_0=\cdots=u_{k-1}=0$. 
The proof is similar to the proof of Lemma \ref{Le:max-reg-con0}. 

We consider the expansion
\[\Bigl( \frac{\delta(\zeta)}\tau -A_q \Bigr)^{-1} = \tau \sum_{n=0}^\infty E_n \zeta^n, 
\qquad |\zeta|<1,\]
which yields (see, e.g., \cite[Section 7]{KLL})
\[u_n=\sum_{j=k}^n\tau E_{n-j}f_j .\]
Proposition \ref{Pr:BDFk} implies
that the map from $(f_n)_{n=k}^N$ to $(A_qu_n)_{n=k}^N$ given by the
formula 
\[A_qu_n=\sum_{j=k}^n\tau A_qE_{n-j}f_j\]
is bounded in $L^p(L^q(\varOmega))$. In other words, 
if we define 
\[w_n:=-\sum_{j=k}^n\tau A_qE_{n-j}(-A_q)^{-1/2}f_j ,\]
then we have 
\begin{equation}\label{w-f-discrete}
\|(w_n)_{n=1}^N\|_{L^p(L^q(\varOmega))} 
\leq C_{p,q}\|((-A_q)^{-1/2}f_n)_{n=1}^N\|_{L^p(L^q(\varOmega))} ,
\end{equation}
where the fractional power operator $(-A_q)^{-1/2}$  commutes with $A_q$.   
It is straightforward to check that 
\[\nabla u_n=\nabla (-A_q)^{-1/2}w_n .\]
Since the Riesz transform $\nabla (-A_q)^{-1/2}$ is bounded on $L^q(\varOmega)$ for 
$q_d'<q<q_d$ (see Appendix), it follows that $\|(\nabla u_n)_{n=1}^N\|_{L^p(L^{q}(\varOmega))}
\leq C_{p,q}\|(w_n)_{n=1}^N\|_{L^p(L^{q}(\varOmega))}$;
therefore, in view of \eqref{w-f-discrete}, we have
\begin{equation}\label{nablaun}
\|(\nabla u_n)_{n=1}^N\|_{L^p(L^{q}(\varOmega))}
\leq C_{p,q}\|((-A_q)^{-1/2}f_n)_{n=1}^N\|_{L^p(L^q(\varOmega))} .
\end{equation}
Moreover, since $(-A_q)^{-1/2}\nabla\cdot\,$ is the dual of the Riesz transform 
$\nabla (-A_q)^{-1/2}$, it is also bounded on $L^q(\varOmega)$ for any $q_d'<q<q_d$, 
and so we have
\begin{align*}
\|((-A_q)^{-1/2}f_n)_{n=1}^N\|_{L^p(L^q(\varOmega))}
&{}
=\|((-A_q)^{-1/2}\nabla\cdot \nabla\varDelta^{-1} f_n)_{n=1}^N\|_{L^p(L^q(\varOmega))}\\ 
&{}\leq \|(\nabla\varDelta^{-1} f_n)_{n=1}^N\|_{L^p(L^q(\varOmega))} ,
\end{align*}
whence
\begin{equation}\label{Aq-12fn}
\|((-A_q)^{-1/2}f_n)_{n=1}^N\|_{L^p(L^q(\varOmega))}
\leq C_{p,q}\|(f_n)_{n=1}^N\|_{L^p(W^{-1,q}(\varOmega))}.
\end{equation}
\eqref{nablaun}--\eqref{Aq-12fn} yield \eqref{BDFconvex}.  
\end{proof}

\section{Sobolev and related inequalities: proof of Lemma~\ref{lem:framework}}
\label{Sec:Sobolev}

A Banach space $X$ is said to be imbedded into  
another Banach space $Y$, denoted by 
$X\hookrightarrow Y$, if 

(a) $u\in X\implies u\in Y$ and this map from $X$ to $Y$ is one-to-one;

(b) $\|u\|_Y\leq C\|u\|_X$ for all $u\in X$,  where $C$ is a constant independent of $u$.\\
The space $X$ is said to be compactly imbedded into $Y$, denoted by 
$X\hookrightarrow\hookrightarrow Y$, 
if in addition to (a)-(b) the following condition is satisfied:

(c) a bounded subset of $X$ is always a pre-compact subset of $Y$. 

\begin{lemma}\label{IntpIneq}
Let $X$, $Y$ and $Z$ be Banach spaces such that 
$X$ is compactly imbedded into $Y$, and 
$Y$ is imbedded into $Z$, i.e.,
\[ X\hookrightarrow\hookrightarrow Y \hookrightarrow Z. \]
%
Then, for any $\varepsilon>0$, there holds
\[\|u\|_Y\leq \varepsilon \|u\|_X + C_\varepsilon\|u\|_Z .\]
\end{lemma}
\begin{proof}
This lemma is probably well known, but since we did not find a reference, we
include the short proof. 

Suppose, on the contrary, that there exists $\varepsilon$ such that 
the inequality above does not hold for all $u\in X$. 
Then there exists a sequence 
$u_n\in X$, $n=1,2,\dotsc,$ such that 
\[\|u_n\|_Y\geq \varepsilon \|u_n\|_X + n\|u_n\|_Z .\]
By a normalization (dividing $u_n$ by a constant), we can 
assume that $\|u_n\|_Y=1$ for all $n\geq 1$. Hence, we have
\[\|u_n\|_X\leq 1/\varepsilon \quad\text{and}\quad\|u_n\|_Z\leq 1/n  .\]

On one hand, since $X$ is compactly embedded into $Y$, 
the boundedness of $u_n$ in $X$ implies the existence of a 
subsequence $u_{n_k}$, $k=1,2,\dotsc,$ which 
converges in $Y$ to some element $u\in Y\hookrightarrow Z$. 
Hence, 
\begin{align}\label{norm_of_u}
\|u\|_Y=\lim_{k\rightarrow \infty}\|u_{n_k}\|_Y=1 . 
\end{align}
On the other hand, $\|u_n\|_Z\leq 1/n$ implies that 
$u_{n_k}$ converges to the zero element in $Z$, 
which means that 
\begin{align}\label{u_is_zero}
u=0 .
\end{align}

Clearly, \eqref{norm_of_u} and \eqref{u_is_zero} contradict 
each other. 
\end{proof}

\begin{lemma}
[Sobolev imbedding]
\label{SobEmbed1} 
For $s>0$, $1<p,q<\infty$ and $d\geq 1$, we have\emph{:} 
\begin{enumerate}[$(1)$]\itemsep=0pt
\item $W^{s,q}(\varOmega)\hookrightarrow\hookrightarrow 
C^{\alpha}(\overline\varOmega)\hookrightarrow L^\infty(\varOmega)$ 
for $\alpha\in(0,s-d/q)$ if $sq>d$, 
\item $W^{s,q}(\varOmega)\hookrightarrow\hookrightarrow 
C^{1,\alpha}(\overline\varOmega)$ 
for $\alpha\in(0,s-1-d/q)$ if $(s-1)q>d$, 
\item $W^{s,p}(\R;X)\hookrightarrow L^{\infty}(\R;X)$ if $sp>1$, 
where $X=L^q(\varOmega)$ or $X=W^{-1,q}(\varOmega)$, 
\item $H^1(\varOmega)\hookrightarrow\hookrightarrow L^{q_0}(\varOmega)$ for all  
$1\leq q_0<2d/(d-2)$ when $d\geq 2$, and 
$q_0=\infty$ when $d=1$.
\end{enumerate}
\end{lemma}

\begin{remark}
(1)--(2) of Lemma \ref{SobEmbed1} are immediate consequences  of 
\cite[p.\ xviii, Sobolev imbedding (18)]{Amann95}; (3) is a simple vector extension of (1);  
(4) can be found in \cite[p.\ 272, Theorem 1]{Evans}. 
\end{remark}

\noindent{\it Proof of (ii) in Lemma~\ref{lem:framework}}. 
For any $1<p,q<\infty$ 
and $r_1,r_2,r\ge 0$ such that $(1-\theta)r_1+\theta r_2=r$ for 
$\theta\in(0,1)$, we denote by 
$B^{r,q}_{p}(\varOmega):=(W^{r_1,q}(\varOmega),W^{r_2,q}(\varOmega))_{\theta,p}$ the 
Besov space of order $r$  (a real interpolation space between two Sobolev spaces, 
see \cite{Tartar07}). Then, 
via Sobolev embedding, we have
\begin{equation*}
\begin{aligned}
&W^{1,p}(0,T;L^q(\varOmega))\cap L^p(0,T;W^{2,q}(\varOmega))  \\
&\hookrightarrow 
L^\infty(0,T;(L^q(\varOmega),W^{2,q}(\varOmega))_{1-1/p,p})
\qquad \mbox{see \cite[Proposition 1.2.10]{Lunardi95}}\\
&\simeq
L^\infty(0,T;B^{2-2/p,q}_{p}(\varOmega)) 
\qquad\qquad\qquad\quad\,\,\, \mbox{according to the definition} \\
&\hookrightarrow 
L^\infty(0,T;C^{1,\alpha}(\varOmega))
\qquad\,\,\,\,\, \mbox{when\,\, $(1-2/p)q>d$\,\,\,$\Leftrightarrow$\,\,\,$2/p+d/q<1$.}
\end{aligned}
\end{equation*} 
and
\begin{equation*}
\begin{aligned}
&W^{1,p}(0,T;W^{-1,q}(\varOmega))\cap L^p(0,T;W^{1,q}(\varOmega))  \\
&\hookrightarrow 
L^\infty(0,T;(W^{-1,q}(\varOmega),W^{1,q}(\varOmega))_{1-1/p,p})
\qquad \mbox{see \cite[Proposition 1.2.10]{Lunardi95}}\\
&\simeq
L^\infty(0,T;B^{1-2/p,q}_{p}(\varOmega)) 
\qquad\qquad\qquad\qquad\quad \mbox{according to the definition} \\
&\hookrightarrow 
L^\infty(0,T;C^{\alpha}(\varOmega))
\qquad\qquad\,\,\,\,\, \mbox{when\,\, $(1-2/p)q>d$\,\,\,$\Leftrightarrow$\,\,\,$2/p+d/q<1$.}
\end{aligned}
\end{equation*} 
\medskip

\noindent{\it Proof of (iii)-(v) in Lemma~\ref{lem:framework}}.  We consider first the setting (P2). 
Property (iii) is standard textbook material. To prove (iv), we note
\[(A(v)-A(w) )u = \nabla\cdot \bigl( (a(v)-a(w)) \nabla u \bigr) = \nabla (a(v)-a(w))\cdot \nabla u 
+  (a(v)-a(w)) \varDelta u\]
and estimate as follows:
%
%
\[\| (A(v)-A(w) )u\|_{L^q(\varOmega)} \le C_R \|v-w\|_{W^{1,\infty}(\varOmega)} \|u\|_{W^{1,q}(\varOmega)}
+C_R \|v-w\|_{L^\infty(\varOmega)} \|u\|_{W^{2,q}(\varOmega)}.\]
%
Since $C^{1,\alpha}(\overline\varOmega)$ is compactly imbedded into 
$W^{1,\infty}(\varOmega)$ and $L^\infty(\varOmega)$, Lemma~\ref{IntpIneq}
gives us the inequalities, for arbitrary $\eps>0$,
\[\begin{aligned}
  \|v-w\|_{W^{1,\infty}(\varOmega)} & \le \varepsilon \|v-w\|_{C^{1,\alpha}(\overline\varOmega)}
  +C_\varepsilon \|v-w\|_{L^2(\varOmega)},\\
 \|v-w\|_{L^\infty(\varOmega)}   \ \  & \le \varepsilon \|v-w\|_{C^{1,\alpha}(\overline\varOmega)}
 +C_\varepsilon \|v-w\|_{L^2(\varOmega)}.
\end{aligned}\]
The  inequality of (v) follows on estimating
%
\[ \langle \varphi,(A(v)-A(w) )u \rangle 
\le C_R \| \varphi \| _{H^1(\varOmega)}\,  
\|v-w\|_{L^{\bar q}(\varOmega)} \|u\|_{W^{1,q}(\varOmega)},\]
where $1/\bar q+1/q=1/2$. Note that for the considered $q>d$ we have 
$\bar q=q/(q/2-1)<d/(d/2-1)$, and so $H^1(\varOmega)$ is compactly imbedded into $L^{\bar q}(\varOmega)$ 
(see (4) of Lemma \ref{SobEmbed1} or \cite[Theorem 6.3]{Adams}), so that by 
Lemma~\ref{IntpIneq}
 \[\|v-w\|_{L^{\bar q}(\varOmega)} \le \eps \|v-w \|_{H^1(\varOmega)} 
 + C_\eps \| v-w \|_{L^2(\varOmega)}.\]

We now turn to the setting (P1). 
Property (iii)  is the same as in (P2), and (v) has actually been shown above. Property (iv) 
follows from estimating
\[ \langle \varphi,(A(v)-A(w) )u \rangle 
\le C_R \| \varphi \| _{W^{1,q'}(\varOmega)}\,  
\|v-w\|_{L^{\infty}(\varOmega)} \|u\|_{W^{1,q}(\varOmega)},\]
where $1/q+1/q'=1$, and from the bound
\[  \|v-w\|_{L^\infty(\varOmega)} \le \eps \| v-w\|_{C^\alpha(\overline\varOmega)} 
+ C_\eps  \| v-w\|_{L^2(\varOmega)},\]
which follows from Lemma~\ref{IntpIneq}.

\appendix
\section{Boundedness of the Riesz transform}

Let $A_q: D(A_q)\rightarrow L^q(\varOmega)$ be defined by  
$A_qu=\nabla\cdot(b(x)\nabla u)$, where 
\[D(A_q)=\{u\in W^{1,q}(\varOmega): 
\nabla\cdot(b(x)\nabla u)\in L^q(\varOmega)\}
\quad\text{and}\quad 
b\in C^\alpha(\overline\varOmega) \,\,\, \text{satisfies } \eqref{ellipt}.\]

\begin{lemma}
The Riesz transform $\nabla (-A_q)^{-1/2}$ is bounded on $L^q(\varOmega)$ 
for $q_d'<q<q_d$, i.e., 
\[\|\nabla (-A_q)^{-1/2}u\|_{L^q(\varOmega)}
\leq C_q\|u\|_{L^q(\varOmega)} ,
\quad \text{for}\,\,\,\,q_d'<q<q_d,\]
where the constant $C_q$ depends on $q$, $\varOmega$, $\alpha$ and 
$\|b\|_{C^\alpha(\overline\varOmega)}$. 
\end{lemma}
\begin{proof}
It has been proved in \cite[Theorem B]{Shen05}
that the Riesz transform $\nabla (-A_q)^{-1/2}$ is bounded on $L^q(\varOmega)$
if and only if the solution of
the homogeneous equation
\begin{equation}\label{HomoEqu}
\nabla\cdot(b(x)\nabla u) = 0
\end{equation}
satisfies the following local estimate:
\begin{equation}\label{HomoEst0}
\bigg(\frac{1}{r^n}\int_{\varOmega\cap B_r(x_0)}
|\nabla u|^q\d x\bigg)^{\frac{1}{q}}
\leq C
\bigg(\frac{1}{r^n}\int_{\varOmega\cap B_{\sigma_0r}(x_0)}
|\nabla u|^2\d x\bigg)^{\frac{1}{2}} ,
\end{equation}
for all $x_0\in\varOmega$ and $0<r<r_0,$ where $r_0$ and $\sigma_0\geq 2$
are any given small positive constants such that $\varOmega\cap B_{\sigma_0r_0}(x_0)$ 
is the intersection of $B_{\sigma_0r_0}(x_0)$ with a Lipschitz graph.
It remains to prove \eqref{HomoEst0}.

Let $\omega$ be a smooth cut-off function which equals
zero outside $B_{2r}:=B_{2r}(x_0)$
and equals 1 on $B_r$.
Extend $u$ to be zero on $B_{2r}\backslash\varOmega$
and denote by $u_{2r}$ the average of $u$ over $B_{2r}$.
Then (\ref{HomoEqu}) implies
\begin{equation}\label{HomoEqu2}
\nabla \cdot (b\nabla (\omega (u-u_{2r})))
=\nabla \cdot (b(u-u_{2r})\nabla \omega )
+b\nabla\cdot\omega \cdot\nabla (u-u_{2r})
\quad\text{in}\,\,\,\varOmega,
\end{equation}
and the $W^{1,q}$ estimate (Lemma \ref{THMJK}) 
implies
\begin{align*}
\|\omega (u-u_{2r})\|_{W^{1,q}(\varOmega)}
&\leq C\|(u-u_{2r})\nabla \omega\|_{L^q(\varOmega)}
+C\|\nabla \omega\cdot\nabla  u\|_{W^{-1,q}(\varOmega)} \\
&\leq C\|(u-u_{2r})\nabla \omega\|_{L^q(\varOmega)} 
+C\|\nabla \omega\cdot\nabla u\|_{L^s(\varOmega)}  \\
&= C\|(u-u_{2r})\nabla \omega\|_{L^q(B_{2r})} 
+C\|\nabla \omega\cdot\nabla  u\|_{L^s(B_{2r})}  \\
&\leq Cr^{-1}\|\nabla u\|_{L^s(B_{2r})} , 
\end{align*}
where $s=qd/(q+d)<q$ satisfies
$L^s(\varOmega)\hookrightarrow W^{-1,q}(\varOmega)$
and $W^{1,s}(\varOmega)\hookrightarrow L^q(\varOmega)$.
The last inequality implies
\begin{equation}\label{HomoEst2}
\|\nabla u\|_{L^q(\varOmega\cap B_{r})}
\leq Cr^{-1}\|\nabla u\|_{L^s(\varOmega\cap B_{2r})} .
\end{equation}
If $s\leq 2,$ then one can derive
\[
\|\nabla u\|_{L^q(\varOmega\cap B_{r})}
\leq Cr^{d/q-d/2}\|\nabla u\|_{L^2(\varOmega\cap B_{2r})} \]
by using once more H\"older's inequality on the right-hand side.
Otherwise, one only needs a finite number of iterations of (\ref{HomoEst2})
to reduce $s$ to be less than $2$. This completes the proof of (\ref{HomoEst0}).
\end{proof}

\subsection*{Acknowledgment}
The research stay of Buyang Li at the University of T\" ubingen
is funded by the Alexander von Humboldt Foundation.

\bibliographystyle{amsplain}

\end{document}